\documentclass[10pt]{amsart}
\usepackage{amsmath,amssymb,xy,array}
\usepackage[T1]{fontenc}
\usepackage{amsfonts}
\usepackage{amsmath}
\usepackage{amssymb}
\usepackage{amsthm}
\usepackage{hyperref}
\usepackage{graphicx}
\usepackage{tikz}
\usepackage{longtable}
\usepackage{geometry}
\usepackage{textcomp}
\usetikzlibrary{trees}

\newtheorem{Theorem}{Theorem}[section]
\newtheorem{theoremn}{Theorem}

\newtheorem{Lemma}[Theorem]{Lemma}
\newtheorem{Proposition}[Theorem]{Proposition}

\theoremstyle{definition}

\newtheorem{Definition}[Theorem]{Definition}
\newtheorem{Remark}[Theorem]{Remark}

\newtheorem{Example}[Theorem]{Example}

\numberwithin{equation}{section}

\newcommand{\arXiv}[1]{\href{http://arxiv.org/abs/#1}{arXiv:#1}}
\def\bibaut#1{{\sc #1}}

\DeclareMathOperator{\Hilb}{Hilb}

\DeclareMathOperator{\mult}{mult}

\DeclareMathOperator{\VSP}{VSP}
\DeclareMathOperator{\Map}{Map}

\DeclareMathOperator{\Ker}{Ker}

\DeclareMathOperator{\Sing}{Sing}

\DeclareMathOperator{\diag}{diag}
\DeclareMathOperator{\Ima}{Im}
\DeclareMathOperator{\expdim}{expdim}
\DeclareMathOperator{\Aut}{Aut}
\DeclareMathOperator{\Bir}{Bir}

\newcommand{\Vgd}{V_g(D)}
\newcommand{\QED}{\ifhmode\unskip\nobreak\fi\quad {\rm Q.E.D.}}

\newcommand{\A}{\mathcal{A}}
\newcommand{\C}{\mathbb{C}}

\newcommand{\OO}{\mathcal{O}}

\newcommand{\pic}{\mathrm{Pic}}

\newcommand{\cG}{\mathcal{G}}

\renewcommand{\P}{\mathbb{P}}

\newcommand{\calh}{\mathcal{H}}

\newcommand{\Z}{\mathbb{Z}}
\newcommand{\HH}{\mathcal{H}}

\renewcommand{\sec}{\mathbb{S}ec}

\newcommand{\hodd}{\mathcal{H}_{1,2d+1}}

\newcommand{\proj}{\mathrm{Proj}}

\begin{document}

\title[\resizebox{3.0in}{!}{Abelian surfaces and theta characteristics}]{Moduli of abelian surfaces, symmetric theta structures and theta characteristics}
\author[Michele Bolognesi]{Michele Bolognesi}
\address{\sc Michele Bolognesi\\
IMAG - Universit\'e de Montpellier\\
Place Eug\`ene Bataillon\\
34095 Montpellier Cedex 5\\ France}
\email{michele.bolognesi@umontpellier.fr}
\author[Alex Massarenti]{Alex Massarenti}
\address{\sc Alex Massarenti\\
Universidade Federal Fluminense\\
Rua M\'ario Santos Braga\\
24020-140, Niter\'oi,  Rio de Janeiro\\ Brazil}
\email{alexmassarenti@id.uff.br}
\date{\today}
\begin{abstract}
We study the birational geometry of some moduli spaces of abelian varieties with extra structure: in particular, with a symmetric theta structure and an odd theta characteristic. For a $(d_1,d_2)$-polarized abelian surface, we show how the parities of the $d_i$ influence the relation between canonical level structures and symmetric theta structures.
For certain values of $d_1$ and $d_2$, a theta characteristic is needed in order to define Theta-null maps. We use these Theta-null maps and preceding work of other authors on the representations of the Heisenberg group to study the birational geometry and the Kodaira dimension of these moduli spaces.
\end{abstract}
\maketitle
\tableofcontents
\section*{Introduction}
Moduli spaces of polarized abelian varieties are one of the subjects with the longest history in algebraic geometry. Very often their study has proceeded along with that of theta functions, in a mingle of analytic and algebraic techniques. Classical results of Tai, Freitag, Mumford and more recent results of Barth \cite{Ba}, O'Grady \cite{O'G}, Gritsenko \cite{Gri1,Gri2}, Gritsenko and Sankaran \cite{GS}, Hulek and Sankaran \cite{HS} agree on the fact that moduli spaces of polarized abelian varieties are very often of general type. Anyway, some exceptions can be found, especially for abelian varieties of small dimension and polarizations of small degree. In these cases the situation has shown to be different and the corresponding moduli spaces are related to beautiful explicit geometrical constructions. For example, the moduli space of principally polarized abelian varieties of dimension $g$ is of general type if $g\geq 7$, and its Kodaira dimension is still unknown for $g=6$. On the other hand the picture is clear for $g\leq 5$. See for instance the work of Katsylo \cite{Kat} for $g = 3$, van Geemen \cite{vG} and Dolgachev-Ortland \cite{DO} for $g = 3$ with a level $2$ structure, Clemens \cite{Cle} for $g = 4$, and Donagi \cite{Do}, Mori-Mukai \cite{MM} and Verra \cite{Ver} for $g = 5$.

Moreover, the geometry of polarized abelian varieties is so rich that one can append many further structures to the moduli functors, obtaining finite covers of the moduli spaces with beautifully intricate patterns, and curious group theory coming into play. One first example of such constructions is the so-called \it level structure \rm (see Section \ref{leveletheta}) which endows the polarized abelian variety with some discrete structure on certain torsion points related to its polarization. In the case of abelian surfaces with a polarization of type $(1,d)$, moduli spaces of polarized abelian surfaces with a level structure have been studied by Gritsenko \cite{Gri1,Gri2}, Hulek and Sankaran \cite{HS}, Gross and Popescu \cite{GP1,GP4,GP2,GP3}, in particular with respect to their birational geometry (rationality, unirationality, uniruledness, and Kodaira dimension) and the general picture seems quite clear. The Kodaira dimension of moduli spaces of $(1,d)$-polarized abelian surfaces has been studied extensively by Sankaran \cite{Sa97}, Erdenberger \cite{Er} and by Hulek, Kahn and Weintraub in \cite{HKW}, where polarizations with level structure are also investigated. In particular, Gritsenko has shown that the moduli space $\A_2(1,d)$ of polarized abelian surfaces of type $(1,d)$ is not unirational if $d \geq 13$ and $d \neq 14, 15, 16, 18, 20, 24, 30, 36.$ Furthermore, thanks again to the results in \cite{Gri1} and \cite{HS} it is now proven that the moduli space of principally polarized abelian surfaces with a level structure $\A_2(1,p)^{lev}$ is of general type for all primes $p\geq 37$.

The aim of this paper is to go a little further in this study of the birational geometry of finite covers of moduli of $(1,d)$-polarized abelian surfaces, concentrating in particular on some spaces that cover finitely the moduli spaces with level structure. In fact, we add to the moduli functor the datum of a symmetric theta structure (see Section \ref{symtheta}), that is an isomorphism of Mumford's Theta group and the abstract Heisenberg group that commutes with the natural involution on the abelian surface. This aspect seems to have been studied quite deeply in the case of a polarization of type $2$, $3$ or $4$ (for instance see \cite{Bo,DL,vdG,SM,SM1,NVG} and \cite{Bo1} for applications to non-abelian theta functions). However, up till now, to the best of our knowledge, it seems to have been ignored for other polarizations. Our study will be mainly aimed at understanding the birational geometry of moduli spaces and will be performed via theta-constant functions. In order to have well-defined theta-constants, it often turns out to be very important to add to our moduli space the choice of a theta characteristic, seen as the quadratic form induced on the points of $2$-torsion by a symmetric line bundle in the algebraic equivalence class of the polarization. For our goals, the choice of the theta characteristic will be equivalent to the choice of the symmetric line bundle. The main results in Section \ref{secquattro} can be summarized as follows.
\begin{theoremn}
Let $d_1,d_2$ be positive integers such that $d_1|d_2$, and let $\mathcal{A}_2(d_1,d_2)^{lev}$ be the moduli space of $(d_1,d_2)$-polarized abelian surfaces with a level structure.
\begin{itemize}
\item[-] If $d_1$ is odd then there exist two quasi-projective varieties $\A_2(d_1,d_2)^{-}_{sym}$ and $\A_2(d_1,d_2)^{+}_{sym}$ parametrizing polarized abelian surfaces with level $(d_1,d_2)$-structure, a symmetric theta structure and an odd, respectively even theta characteristic. Furthermore, there are natural morphisms 
$$f^{-}:\A_2(d_1,d_2)^{-}_{sym}\rightarrow\mathcal{A}_2(d_1,d_2)^{lev}, \quad f^{+}:\A_2(d_1,d_2)^{+}_{sym}\rightarrow\mathcal{A}_2(d_1,d_2)^{lev}$$ 
forgetting the theta characteristic. If $d_1,d_2$ are both odd then $f^{-}$ and $f^{+}$ have degree $6$ and $10$ respectively. If $d_2$ is even then $f^{-}$ has degree $4$, while $f^{+}$ has degree $12$.
\item[-] If $d_1$ and $d_2$ are both even then there exists a quasi-projective variety $\A_2(d_1,d_2)_{sym}$ parametrizing polarized abelian surfaces with level $(d_1,d_2)$-structure, and a symmetric theta structure. Furthermore, there is a natural morphism 
$$f:\A_2(d_1,d_2)_{sym}\rightarrow\mathcal{A}_2(d_1,d_2)^{lev}$$ 
of degree $16$ forgetting the theta structure.
\end{itemize}
\end{theoremn}
In this paper we concentrate on abelian surfaces with an odd theta characteristic and on the moduli spaces $\A_2(1,d)^{-}$. The case of even theta characteristic will be addressed in our forthcoming paper \cite{BM}. The structure of $\A_2(1,d)^{-}$ is slightly different depending on whether $d$ is even and $f^{-}$ has degree $4$, or $d$ is odd and $f^{-}$ has degree $6$.\\
In Section \ref{seccinque} we study the birational geometry of these moduli spaces using objects and techniques coming from birational projective geometry such as varieties of sums of powers, conic bundles, and the Segre criterion for the unirationality of smooth quartic $3$-folds.\\ 
Our main results in Theorems \ref{A17}, \ref{A19}, \ref{12}, \ref{18}, Propositions \ref{A111}, \ref{A116}, and Paragraphs \ref{A113}, \ref{10}, \ref{1416}, can be summarized as follows.
\begin{theoremn}
Let $\A_2(1,d)^{-}_{sym}$ be the moduli space of $(1,d)$-polarized abelian surfaces, endowed with a symmetric theta structure and an odd theta characteristic. Then
\begin{itemize}
\item[-] $\mathcal{A}_2(1,7)^{-}_{sym}$ is birational to the variety of sums of powers $\VSP_6(F,6)$ (see Definition \ref{vsph}), where $F\in k[x_0,x_1,x_2]_4$ is a general quartic polynomial. In particular $\mathcal{A}_2(1,7)^{-}_{sym}$ is rationally connected.
\item[-] $\mathcal{A}_2(1,9)^{-}_{sym}$ is rational.
\item[-] $\mathcal{A}_2(1,11)^{-}_{sym}$ is birational to a sextic pfaffian hypersurface in $\mathbb{P}^4$, which is singular along a smooth curve of degree $20$ and genus $26$. 
\item[-] $\mathcal{A}_2(1,13)^{-}_{sym}$ is birational to a $3$-fold of degree $21$ in $\mathbb{P}^5$, which is scheme-theoretically defined by three sextic pfaffians. 
\item[-] $\mathcal{A}_2(1,8)_{sym}^{-}$ is birational to a conic bundle over $\mathbb{P}^2$ whose discriminant locus is a smooth curve of degree $8$. In particular, $\mathcal{A}_2(1,8)_{sym}^{-}$ is unirational but not rational.
\item[-] $\mathcal{A}_2(1,10)_{sym}^{-}$ is rational.
\item[-] $\mathcal{A}_2(1,12)_{sym}^{-}$ is unirational but not rational.
\item[-] $\mathcal{A}_2(1,14)^{-}_{sym}$ is birational to a $3$-fold of degree $16$ in $\mathbb{P}^5$, which is singular along a curve of degree $24$ and scheme-theoretical complete intersection of two quartic pfaffians.
\item[-] $\mathcal{A}_2(1,16)^{-}_{sym}$ is birational to a $3$-fold of degree $40$, and of general type in $\mathbb{P}^6$. 
\end{itemize}
\end{theoremn}
\subsection*{Plan of the paper}
In Section \ref{secdue} we introduce most of our base notation and make a quick summary of the results we will need about level structures, the Theta and Heisenberg group, theta structures, theta characteristics and quadratic forms on $\Z/2\Z$-vector spaces. Section \ref{sectre} is devoted to the study of linear systems on abelian surfaces. Since we need an intrinsic way to compute the dimension of the spaces of sections for the objects of our moduli spaces, we make use of the Atiyah-Bott-Lefschetz fixed point formula, and deduce these dimensions for different choices of the line bundle representing the polarization. The goal of Section \ref{secquattro} is the construction of the arithmetic groups that define our moduli spaces as quotients of the Siegel half-space $\mathbb{H}_2$. Once these subgroup are defined, we display the theta-constant maps that yield maps to the projective space. These maps, and their images, are studied in Section \ref{seccinque}, by tools of projective and birational geometry, and several results about the birational geometry and Kodaira dimension of $\A_2(d_1,d_2)^{-}_{sym}$ are proven.
\subsubsection*{Acknowledgments}
First of all, we want to heartfully thank the anonymous referee, who corrected some mistakes, suggested Remark \ref{level-image} and hugely helped us to bring this paper to a better form.
We gratefully acknowledge \textit{G. van der Geer}, \textit{I. Dolgachev}, \textit{K. Hulek}, \textit{C. Ritzenthaler}, \textit{G. Sankaran}, \textit{M. Gross}, \textit{N.Shepherd-Barron}, and especially \textit{B. van Geemen} and \textit{R. Salvati Manni} for fruitful conversations and observations.\\
The authors are members of the Gruppo Nazionale per le Strutture Algebriche, Geometriche e le loro Applicazioni of the Istituto Nazionale di Alta Matematica "F. Severi" (GNSAGA-INDAM). This work was done while the second named author was a Post-Doctorate at IMPA, funded by CAPES-Brazil. The first named author is member of the GDR GAGC of the CNRS.
\section{Notation and Preliminaries}\label{secdue}
The main references for this section are \cite{BL} and \cite{HKW}. Let $A$ be an abelian variety of dimension $g$ over the complex numbers. The variety $A$ is a quotient $V/\Lambda$, where $V$ is a $g$-dimensional complex vector space and $\Lambda$ a lattice. Let $L$ be an ample line bundle on $A$, and let us denote by $H$ the corresponding \it polarization, i.e. \rm the first Chern class of $L$. We denote by $\pic^H(A)$ the set of line bundles whose polarization is $H$. The polarization $H$ induces a positive-definite Hermitian form, whose imaginary part $E := Im (H)$ takes integer values on the lattice $\Lambda$. There exists a natural map from $A$ to its dual, $\phi_L:A\to \hat{A}$, defined by $L$ as $x\mapsto t_x^*L\otimes L$, where $t_x$ is the translation in $A$ by $x$. We denote the kernel of $\phi_L$ by $K(L)$. It always has the form $K(L) = (\Z/d_1\Z \oplus \cdots \oplus \Z/d_g\Z)^{\oplus 2}$, where $d_1|d_2| \cdots |d_g$. The ordered $g$-tuple $D = (d_1, \ldots , d_g)$ 
is called the \it type \rm of the polarization. For sake of shortness, we will write $\Z^g/D\Z^g$ for $\Z/d_1\Z \oplus \cdots \oplus \Z/d_g\Z$. The form $E$ defines the \it Weil pairing \rm on $K(L)$ as $e^H(x, y) := \exp(2\pi iE(x, y))$ for $x, y \in K(L)$. A decomposition of the lattice $\Lambda=\Lambda_1 \oplus \Lambda_2$
is said to be a decomposition for $L$ if $\Lambda_1$ and $\Lambda_2$ are isotropic for $E$. This induces a decomposition of real vector spaces $V = V_1 \oplus V_2.$ Let us now define $\Lambda(L):=\{v\in V\: |\: E(v,\Lambda)\subset \Z\}$. Since $K(L)=\Lambda(L)/\Lambda$, a decomposition of $\Lambda$ also induces a decomposition 
\begin{equation}\label{deco}
K(L)=K_1(L)\oplus K_2(L),
\end{equation}
where both subgroups are isotropic with respect to the Weil pairing and are isomorphic to $(\Z^g/D\Z)^g$.
\subsubsection{Theta characteristics}
Let $(A,H)$ be a polarized abelian variety and let $\imath:A \to A$ be the canonical involution. A line bundle $L$ is symmetric if $\imath^*L\cong L$.
If $L$ is symmetric, a morphism $\varphi:L \to L$ is called an \it isomorphism of $L$ over $\imath$ \rm if it 
commutes with $\imath$ for every $x\in A$, and the induced map $\varphi(x):L(x) \to L(-x)$ is $\mathbb{C}$-linear. The isomorphism is \it normalized \rm if $\varphi(0)$ is the identity. The following result is well known, \cite[Section 2]{Mum1}, \cite[Lemma 4.6.3]{BL}.
\begin{Lemma}\label{uniquenormalized}
Any symmetric line bundle $L\in \pic(A)$ admits a unique normalized isomorphism $\varphi:L\to L$ over $\imath$.
\end{Lemma}
We will denote by $A[n]$ the set of $n$-torsion points of the abelian variety $A$. Our next goal is to define theta characteristics via the theory of quadratic forms over the $\Z/2\Z$-vector space $A[2]$. Given a polarization $H\in NS(A)$, we define a symmetric bilinear form $q^H: A[2]\times A[2] \to \{\pm 1\}$ by
$q^H(v,w):=\exp(\pi i E(2v,2w))$.
\begin{Definition}
A \it theta characteristic \rm is a quadratic form $q:A[2]\to \{\pm 1\}$ associated to $e^H$, that is:
$$q(x)q(y)q(x+y)=q^H(x,y),$$
for all $x,y\in A[2]$.
\end{Definition}
We denote the set of theta characteristics by $\vartheta(A)$. Every symmetric line bundle $L$ 
defines a theta characteristic as follows. 
\begin{Definition}\label{bijtheta}
Let $L\in\pic^H(A)$ be a symmetric line bundle, and $x \in A[2]$. We define $e^L(x)$ as the scalar $\beta$ such that
$\varphi(x):L(x) \stackrel{\sim}{\to} (\imath^*L)(x)=L(\imath(x))=L(x)$
is multiplication by $\beta$. 
\end{Definition}
Let $D$ be the symmetric divisor on $A$ such that $L\cong \OO_A(D)$. The quadratic form $e^L$ can be also defined as follows:
\begin{equation}\label{thetadiv}
e^L(x):=(-1)^{\mult_x(D)-\mult_0(D)}.
\end{equation}
From \cite[Lemma 4.6.2]{BL} one sees that the set of theta characteristics for an abelian surface is a torsor under the action of $A[2]$ and hence it has cardinality $16$. For a nice and general introduction to the finite geometries in the theory of theta characteristics, the reader may check
\cite{Sa76}.
\subsubsection{Level structures and theta structures}\label{leveletheta}
If $x\in K(L)$, then $x$ induces an isomorphism $t_x^*L\cong L$, so we get a representation
$K(L)\to PGL(H^0(A,L)).$ This representation does not come from a linear representation of $K(L)$, but it lifts to a linear representation of the central extension of $K(L)$ defined by the following exact sequence:
$$1 \to \C^* \to \cG(L) \to K(L) \to 0.$$
The commutator of $\cG(L)$ is exactly the Weil pairing $e^H$. The group $\cG(L)$ is called the \it theta group \rm of $L$. As an abstract group, $\cG(L)$ is isomorphic to the Heisenberg group $\calh(D)$ of type $D$. The group $\calh(D)$ as a set is equal to $\C \times K(D)$, where $K(D)= \Z^g/D\Z^g \oplus \Z^g/D\Z^g$. The group structure of $\calh(D)$ is defined as follows. Let $f_1,\dots, f_{2g}$ be the standard basis of $K(D)$. We define an alternating form $e^D:K(D) \times K(D) \to \C^*$ on this basis as follows:
\begin{equation}\label{heis-pairing}
e^D(f_\alpha,f_\beta):=\left\{
\begin{array}{lr}
\exp(-2\pi i/d_{\alpha}) & \mathrm{if}\ \beta= g + \alpha, \\
\exp(2\pi i/d_{\alpha}) & \mathrm{if}\ \alpha= g + \beta, \\
1 & \mathrm{otherwise}.\end{array}
\right.
\end{equation}
The group structure of $\calh(D)$ is defined via $e^D$. Given $(a,x_1,x_2),(b,y_1,y_2) \in \calh(D)$ we have
$(a,x_1,x_2),(b,y_1,y_2):= (abe^D(x_1,y_2),x_1+y_1,x_2+y_2).$
Similarly to the case of the theta group, the Schur commutator is given by the pairing $e^D$. An isomorphism 
$\theta:\cG(L) \stackrel{\sim}{\to} \calh(D)$ 
that restricts to the identity on $\C^*$ is called a \textit{theta structure}. Any theta structure induces a symplectic isomorphism between $K(L)$ and $K(D)$, with respect to the alternating forms $e^L$ and $e^D$. A symplectic isomorphism $K(L)\stackrel{\sim}{\to}K(D)$ is traditionally called a \it level-D structure \rm (of canonical type).\\
As we have already observed, the theta group has a natural representation $\rho:\cG(L)\to GL(H^0(A,L))$ which lifts in a unique way the representation $K(L) \to PGL(H^0(A,L))$. The choice of a theta structure induces an isomorphism between $\rho$ and a certain representation of $\calh(D)$ called Schr\"{o}dinger representation. Let us outline its construction.
Let $V_g(D):=\Map(\Z^g/D\Z^g,\C)$ be the vector space of complex functions defined on the set $\Z^g/D\Z^g$. The Schr\"{o}dinger representation $\sigma:\calh(D)\to GL(V_g(D))$ is irreducible and defined as follows:
$\sigma(\alpha,a,b)(v):=\alpha e^D(-,b)v(-+a).$
The center $\C^*$ clearly acts by scalar multiplication, hence $\sigma$ induces a projective representation of $K(D)$. If $A$ is a surface and $D=(d_1,d_2)$, a basis of $V_2(D)$ is given by the the functions $\delta_x$, for $x\in \Z^2/D\Z^2$, defined by $\delta_x(y):= \delta_{xy}$, where $\delta_{xy}$ is the Kronecker delta 
\begin{eqnarray}\label{deltas}
\delta_{xy} & := & \left\{\begin{array}{lr}
1 & \mathrm{if}\ x=y, \\
0 & \mathrm{otherwise}.\\
\end{array}
\right.
\end{eqnarray}
Given an ample line bundle $L$ and a decomposition for $L$, there is a unique basis $\{\vartheta_x\:|\: x\in K_1(L)\}$ (see the decomposition in equation (\ref{deco})) of \it canonical theta functions \rm of the space $H^0(A,L)$ \cite[Section 3.2]{BL}. Hence, a canonical basis of theta functions, indexed by $K_1(L)\cong \Z^2/D\Z^2$, for $H^0(A,L)$ yields an identification of $H^0(A,L)$ and $V_2(D)$ such that the two representations $\calh(D)\to GL(V_g(D))$ and $\cG(L) \to GL(H^0(A,L))$ coincide. The projective image of $A$ in $\P(V_g(D))$ will be equivariant under the Schr\"{o}dinger representation, and also all the spaces $H^0(A,\mathcal{I}_A(n))$ will be representations of the Heisenberg group. It will be useful for the rest of the paper to define the \textit{finite Heisenberg group}. 
\begin{Definition}\label{finiteheis}
We will denote by $\calh_{d_1,d_2}$ the subgroup of $\calh(D)$ generated by $\sigma_1 = (1, 1, 0, 0, 0),\ \sigma_2 = (1, 0, 1, 0, 0),\
\tau_1 = (1, 0, 0, 1, 0)$ and $\tau_2 = (1, 0, 0, 0, 1)$. Let $x=(i,j)\in\Z^2/D\Z^2$, the elements $\sigma_i$ and $\tau_i$ act on $V_2(D)$ via
$$
\begin{array}{ll}
\sigma_1(\delta_{(i,j)}) = \delta_{(i-1,j)}, & \sigma_2(\delta_{(i,j)}) = \delta_{(i,j-1)}, \\ 
\tau_1(\delta_{(i,j)}) = \xi_1^{-i}\delta_{(i,j)}, & \tau_2(\delta_{(i,j)}) =\xi_2^{-j}\delta_{(i,j)},
\end{array} 
$$
where $\xi_k := \exp(2\pi i/d_k)$.
\end{Definition}
In particular, if $d_1 = 1$, $\sigma_1$ and $\tau_1$ act both as the identity, so for shortness we will denote by $\sigma$ and $\tau$ the generators $\sigma_2$ and $\tau_2$, and not consider the first index on the variables.
\subsubsection{Symmetric theta structures}\label{symtheta}
Whenever we talk about symmetric theta structures, we will implicitly assume that $L$ is a symmetric line bundle. First of all, recall that $K(L)$ acts on $A$ via translations. In turn, the involution $\imath$ acts on $K(L)$ as $-1$. Hence we can define the extended group $K(L)^e:=K(L)\rtimes \imath$ and the extended theta group $\cG(L)^e$ as a central extension of $K(L)^e$ by $\C^*$. More precisely we set $\cG(L)^e:=\cG(L)\rtimes \imath_L$, where $\imath_L$ is the obvious extension of $\imath_L$ to $\cG(L)$ acting as the identity on $\C^*$.
In a similar way, we introduce the extended Heisenberg group $\calh(D)^e:=\calh(D)^e\rtimes \imath_D$, where $\imath_D(z,x_1,x_2)=(z,-x_1,-x_2)$. By \it extended theta structure \rm we mean an isomorphism of $\calh(D)^e$ with $\cG(L)^e$ inducing the identity on $\C^*$. Any extended theta structure induces a theta structure, but on the other hand a theta structure $\theta$ can be extended if and only if it is a \textit{symmetric theta structure}, that is if 
$\theta \circ \imath_L = \imath_D\circ \theta$.

In particular, the Schr\"{o}dinger representation $\rho$ extends to a representation $\rho^e$ of $\calh(D)^e$. When $A$ is a surface the action of $\imath_D$ is $\rho^e(\imath_D)(\delta_{(i,j)})=\delta_{(-i,-j)}$. The involution $\imath_D$ acts on the space $V_2(D)$ spanned by delta functions and decomposes it into an invariant and an anti-invariant eigenspace. We will denote by $\P^{\mathfrak{n}}_+$ and $\P^{\mathfrak{m}}_-$ the corresponding projective spaces. The dimensions $\mathfrak{n}$ and $\mathfrak{m}$ will be computed in the next section.
If $D=(1,d)$, then $\P^{\mathfrak{n}}_+$ is given by the equations $x_i=x_{-i}$, for $i\in \Z/d\Z$, and $\P^{\mathfrak{m}}_-$  by the equations $x_i=-x_{-i}$, for $i$ in the same range.
\begin{Definition}
Let $\Aut(\calh(D))$ be the group of automorphisms of the Heisenberg group $\calh(D)$. We will denote 
$$\Aut_{\mathbb{C}^{*}}(\mathcal{H}(D)):= \{\phi \in  \Aut(\calh(D))\: | \: \phi(t,0,0)=(t,0,0), \: \forall t\in \C^*\}.$$
\end{Definition}
The set of all theta structures for a line bundle $L$ of type $D$ is a principal homogeneous space under the action of $\Aut_{\mathbb{C}^{*}}(\mathcal{H}(D))$. Let $Sp(D)$ denote the group of all automorphisms of $K(D)$ that preserve the alternating form $e^D$. The set of all level $D$ structures is a principal homogeneous space for the group $Sp(D)$. From \cite[Lemma 6.6.3]{BL} one sees that any element of $\Aut_{\mathbb{C}^{*}}(\mathcal{H}(D))$ induces a symplectic automorphism of $K(D)$. Moreover, for all $z\in K(D)$ we define an element $\gamma_z(\alpha, x_1,x_2):=(\alpha e^D(z,x_1+x_2),x_1,x_2)\in \Aut_{\mathbb{C}^{*}}(\mathcal{H}(D))$. This yields an injective homomorphism $\gamma:K(D) \to \Aut_{\mathbb{C}^{*}}(\mathcal{H}(D))$. From \cite[Lemma 6.6.6]{BL} we obtain the following
\begin{Lemma}\label{sequenza}
There exists an exact sequence 
\begin{equation}\label{aut-sequence}
1\to K(D) \stackrel{\gamma}{\to} \Aut_{\mathbb{C}^{*}}(\mathcal{H}(D)) \to Sp(D) \to 1.
\end{equation}
\end{Lemma}
\begin{Remark}\label{aut-action} 
If $\varphi \in \Aut_{\mathbb{C}^{*}}(\mathcal{H}(D))$, then $\sigma \circ \varphi$ is also an irreducible level one representation, that is a central element $z\in\C^*$ acts by multiplication with itself. Hence by the Schur lemma there exists a unique linear map $G_{\varphi} : V_2(D) \to V_2(D)$, such that $G_{\varphi}(\sigma(h)) = \sigma(\varphi(h))$ for all $h \in \calh(D)$. In this way we obtain a representation
\begin{equation}\label{Gtilde}
\begin{array}{ccc}
\widetilde{G}: \Aut_{\mathbb{C}^{*}}(\mathcal{H}(D)) & \rightarrow & GL(\Vgd)\\ 
\varphi & \mapsto & G_{\varphi}.
\end{array} 
\end{equation}
\end{Remark}
\begin{Lemma}\label{decogtilde}
Let $C_{\imath^D}\subset \Aut_{\mathbb{C}^{*}}(\mathcal{H}(D))$ be the centralizer subgroup of $\imath^D$, $\Vgd^+$ and $\Vgd^-$ the eigenspaces of $\Vgd$ with respect to the standard involution on $(\Z/D\Z)^g$. Then the restriction of the representation $\widetilde{G}$ to $C_{\imath^D}$ splits into two representations $\widetilde{G}^+: C_{\imath^D}  \rightarrow  GL(\Vgd^+),$ and
$\widetilde{G}^-: C_{\imath^D}  \rightarrow  GL(\Vgd^-).$
\end{Lemma}
\begin{Theorem}\label{symmetry}\cite[Theorem 6.9.5]{BL}
Let $A$ be an abelian variety of dimension $g$ and $H$ be a polarization of type $D=(d_1,\dots,d_g)$ with $d_1,\dots,d_s$ odd and $d_{s+1},\dots,d_g$ even. There are $2^{2s}$ symmetric line bundles in $\pic^H(A)$, each admitting exactly $2^{2(g-s)}\cdot\#(Sp(D))$ symmetric theta structures. The remaining symmetric line bundles in $\pic^H(A)$ do not admit any symmetric theta structure.
\end{Theorem}
\begin{Remark}\label{remsymmetry}
More precisely, if $L\in \pic^H(A)$, the symmetric theta structures inducing a given $D$-level structure correspond to elements of
$K(L)\cap A[2]\cong (\Z/2Z)^{2(g-s)}.$
On the other hand, the symmetric line bundles admitting symmetric theta structures are represented by elements of
$A[2]/(K(L)\cap A[2]) \cong (\Z/2\Z)^{2s}.$
\end{Remark}
From Theorem \ref{symmetry} and Remark \ref{remsymmetry} we obtain a straightforward version of Lemma \ref{sequenza} for symmetric theta structures.
\begin{Lemma}\label{symsequenza}
There exists an exact sequence 
\begin{equation}\label{autsym-sequence}
1\to K(D)\cap A[2] \to C_{\imath^D} \to Sp(D) \to 1
\end{equation}
\end{Lemma}
\section{Theta characteristics and linear systems on abelian surfaces}\label{sectre}
Let $(A,H)$ be a $(d_1,d_2)$-polarized abelian surface and $L\in\pic^H(A)$ a symmetric line bundle. The normalized isomorphism induces an involution $\imath^{\#}:H^0(A,L) \to H^0(A,L)$ defined by $\imath^{\#}(s)=\imath^*(\varphi(s))$.
In the rest of the paper we will need an intrinsic computation of the dimensions of the eigenspaces $H^0(A,L)^+$ and $H^0(A,L)^-$. In the same spirit of \cite[Section 2.1]{Bo} we will compute this via the Atiyah-Bott-Lefschetz fixed point formula \cite[Page 421]{GH}). Suppose that $L$ admits a symmetric theta structure (see Theorem \ref{symmetry}). Denote by $A[2]^+$ (respectively $A[2]^-$) the set of $2$-torsion points where $e^L$ takes the value $+1$ (respectively $-1$). 
\begin{Proposition}\label{actions}
Let $H$ be a polarization of type $(d_1,d_2)$, with $d_1|d_2$ as usual, and $L\in \pic^H(A)$ a symmetric line bundle admitting a symmetric theta structure. Then for the theta characteristic $e^L$, we have:
\begin{itemize}
\item[-] if both $d_1$ and $d_2$ are odd then $\#(A[2]^+)=10$ and $\#(A[2]^-)=6$ (in which case we say that $e^L$ is an even theta characteristic), or $\#(A[2]^+)=6$ and $\#(A[2]^-)=10$ (in which case we say that $e^L$ is an odd theta characteristic), 
\item[-] if $d_1$ is odd and $d_2$ is even then $\#(A[2]^+)=12$ and $\#(A[2]^-)=4$ ($e^L$ is an even theta characteristic), or $\#(A[2]^+)=4$ and $\#(A[2]^-)=12$ ($e^L$ is an odd theta characteristic),
\item[-] if both $d_1$ and $d_2$ are even, then $\#(A[2]^+)=16$ and $\#(A[2]^-)=0$, for all theta characteristics.
\end{itemize}
\end{Proposition}
Proposition \ref{actions} follows immediately from \cite[Proposition 4.7.5]{BL}. We will simply say that a symmetric line bundle is odd (respectively even) if it induces an odd (respectively even) theta characteristic $e^L$.
\begin{Proposition}\label{dim-e-bl}
Let $A$ be an abelian surface and $L$ a symmetric line bundle inducing a polarization of type $(d_1,d_2)$ on $A$, and admitting a symmetric theta structure.
\begin{enumerate}
\item If $d_1,d_2$ are odd, and if $L$ is even then $h^0(A,L)^+=\frac{d_1d_2+1}{2}$, $h^0(A,L)^-=\frac{d_1d_2-1}{2}$; if $L$ is odd then $h^0(A,L)^-=\frac{d_1d_2+1}{2}$, $h^0(A,L)^+=\frac{d_1d_2-1}{2}$.
\item If $d_1$ is odd and $d_2$ even, and if $L$ is even then $h^0(A,L)^+=\frac{d_1d_2}{2}+1$, $h^0(A,L)^-=\frac{d_1d_2}{2}-1$; if $L$ is odd then $h^0(A,L)^-=\frac{d_1d_2}{2}+1$, $h^0(A,L)^+=\frac{d_1d_2}{2}-1$.
\item If $d_1$ and $d_2$ are even, whatever the parity of $L$, we have $h^0(A,L)^+=\frac{d_1d_2}{2}+2$, $h^0(A,L)^-=\frac{d_1d_2}{2}-2$.
\end{enumerate}
Moreover, whatever the parity of $L$, in the first two cases, the base locus of the invariant linear system is $A[2]^-$ (hence $A\cap\P (H^0(A,L)^+)^* = A[2]^+$), and the base locus of the anti-invariant linear system is $A[2]^+$ (hence $A\cap\P (H^0(A,L)^-)^* = A[2]^-$). By definition $0 \in A[2]^+$. When both the coefficients are even, $H^0(A,L)^+$ is base point free, the base locus of $H^0(A,L)^-$ is $A[2]$ and hence $A\cap\P(H^0(A,L)^+)^* = A[2]$.
\end{Proposition}
\begin{proof}
We will use the Atiyah-Bott-Lefschetz fixed point formula. The fixed points of $\imath$ are precisely the $2$-torsion points, hence the formula gives 
$$\sum_{j=0}^2(-1)^j Tr(\imath^{\#}:H^j(A,L))=\sum_{\alpha \in A[2]} \frac{Tr(\imath:L_{\alpha} \to L_{\alpha})}{\det(Id-(d\imath)_\alpha)}.$$
Now we remark that $(d\imath)= -Id$, hence $\det (Id-(d\imath)_\alpha) = 4$ for all $\alpha\in A[2]$. Now, if $L$ is even
\begin{equation}
\sum_{\alpha \in A[2]} Tr(\imath: L_{\alpha} \to L_{\alpha})= 
\left\{ \begin{array}{cc} 4 & \mathrm{in\ case\ (1)},\\ 
8 & \mathrm{in\ case\ (2)},\\
 16 & \mathrm{in\ case\ (3)}. 
\end{array}\right.
\end{equation}
If $L$ is odd then these quantities equal $-4,-8$ and $16$, respectively. Then, we observe that $h^p(A,L)=0$ for $p>0$ by Kodaira vanishing. By the definition of the eigenspaces, this in turn means that 
$\sum_{j=0}^2(-1)^j Tr(\imath^{\#}:H^j(A,L))=h^0(A,L)^+-h^0(A,L)^-.$
This implies that, if $L$ is an even line bundle representing $H$, we have
\begin{equation}
\begin{array}{llc}
h^0(A,L)^++h^0(A,L)^- & = & d_1d_2,\\
h^0(A,L)^+-h^0(A,L)^- & = & \left\{\begin{array}{cc} 1 & \mathrm{in\ case\ (1)},\\ 2 & \mathrm{in\ case\ (2)},\\ 4 & \mathrm{in\ case\ (3)}. \end{array}\right.
\end{array}
\end{equation}
which implies the claim. On the other hand, if $L$ is an odd line bundle representing $H$, then
\begin{equation}
\begin{array}{llc}
h^0(A,L)^++h^0(A,L)^- & = & d_1d_2,\\
h^0(A,L)^+-h^0(A,L)^- & = &\left\{\begin{array}{cc} -1 & \mathrm{in\ case\ (1)},\\ -2 & \mathrm{in\ case\ (2)},\\ 4 & \mathrm{in\ case\ (3)}. \end{array}\right.
\end{array}
\end{equation}
hence in the first two cases the dimensions of the eigenspaces are interchanged, and in the third case they stay the same.\\
Let us now come to the base locus. The same argument works for the three cases. The union of the base loci of $H^0(A,L)^+$ and $H^0(A,L)^-$ is $A[2]$, and by definition of normalized isomorphism the origin is contained in $A[2]^+$. Given an invariant section $s\in H^0(A,L)_+$ and a 2-torsion point $z\in A[2]^-$, by definition of $e^L$, we have $s(z)=(\imath^{\#}(s))(z)=-s(z)$, and thus $s$ vanishes at $z$. The same argument shows that all anti-invariant sections vanish at points of $A[2]^+$. The claims about the intersections of $A$ with the eigenspaces are a straightforward consequence of those about the base loci.
\end{proof}
\section{The arithmetic groups for moduli of abelian surfaces with symmetric theta structure}\label{secquattro}
This section is devoted to the construction of the arithmetic groups that are needed in order to construct moduli spaces of polarized abelian surfaces, endowed with a symmetric theta structure plus a symmetric line bundle representing the polarization, as quotients of the Siegel half-space $\mathbb{H}_2$. Theta characteristics are reasonably manageable  group theoretically, since $Sp_{4}(\Z/2\Z)$ (the reduction modulo $2$ of the modular group $Sp_{4}(\Z)$) naturally acts on quadratic forms on a $4$-dimensional $\Z/2\Z$-vector space. For simplicity we state these results only for abelian surfaces but the same proofs give analogous statements for abelian varieties of any dimension and polarization type. First, we want to study the action of arithmetic subgroups on subsets or quotients of $A[2]$. In order to do this, we need to introduce half-integer characteristics.

A half-integer characteristic $m$ is an element of $(\frac{1}{2}\Z^4/\Z^4)$. The set $A[2]$ of 2-torsion points is in bijection (non-canonically) with the set of half-integer characteristics \cite[Section 2]{Igu}. Moreover, the natural action of $Sp_{4}(\Z)$ on $\mathbb{H}_2$ induces a transformation formula for theta functions with half-integer characteristics \cite[Section 2]{Igu}. The zero loci of theta functions with half-integer characteristics are symmetric theta divisors. These divisors define in turn quadratic forms on $A[2]$ via the identification (\ref{thetadiv}), thus yielding a (non-canonical) bijection between half-integer characteristics and $\vartheta(A)$. The action on theta functions induces an action on half-integer characteristics, which is given by taking the formula
\begin{equation}\label{char-action}
M\cdot 
\left(\begin{array}{c} a \\ b \end{array}\right) = \left( \begin{array}{cc} \textbf{D} & -\textbf{C} \\ -\textbf{B} & \textbf{A} \end{array}\right)\left(\begin{array}{c} a \\ b \end{array}\right)+\left(\begin{array}{c}
\diag(\textbf{CD}^t)\\ \diag(\textbf{AB}^t) \end{array}\right),
\end{equation}
modulo $2$, for $M\in Sp_{4}(\Z)$ and $a,b\in (\frac{1}{2}\Z^2/\Z^2)$.
\subsubsection{The odd case}
The goal of this subsection is to show how, when $d_1$ and $d_2$ are odd, a level $D$ structure induces uniquely a symmetric theta structure. In the odd case, the claim of Lemma \ref{sequenza} is even simpler. 
\begin{Lemma}\label{sequence-split}
The exact sequence (\ref{aut-sequence}) splits.
\end{Lemma}
\begin{proof}
The proof of this Lemma follows closely the arguments of Sect. 2.2 of \cite{Bo}. We just remark that
the centralizer subgroup $C_{\imath^D}\subset \Aut_{\mathbb{C}^{*}}(\mathcal{H}(D))$ of $\imath^D$ is the section from $Sp(D)$ that make the sequence split.
\end{proof}
Moreover, by Remark \ref{remsymmetry}, we have $16$ such bundles that we identify with theta characteristics by taking their associated quadratic forms on $2$-torsion points.  The exact sequence (\ref{autsym-sequence}) reduces to an isomorphism
$C_{\imath^D}\cong Sp(D)$, so the symmetric theta structure is completely determined once the line bundle is chosen, and the action of $Sp(D)$ on the line bundles corresponds to the action (\ref{char-action}) on the half-integer characteristics.
\subsubsection{The congruence subgroups in the odd case}\label{cong}
We will denote by $M_g(\Z)$ the space of $g\times g$ matrices with entries in $\Z$ and by $\Gamma_g$ the symplectic group $Sp_{2g}(\mathbb{Z})$. We will now introduce arithmetic subgroups of $\Gamma_2$ that are extensions of subgroups of $Sp_4(\mathbb{Z}/2\mathbb{Z})$. Similar groups have been described in \cite[Chapter 1]{HKW}. As is customary, we will denote by $\Gamma_2(d)$ the level $d$ subgroup, that is the kernel of the reduction modulo $d$ morphism 
$r_d:Sp_4(\mathbb{Z})\rightarrow Sp_4(\mathbb{Z}/d\mathbb{Z})$. The following result is probably well known to experts of the field, but we haven't been able to find a reference.
\begin{Lemma}\label{l1}
Let $d$ be an odd integer. Then we have the following exact sequence
$$1\to \Gamma_2(2d)\xrightarrow{i}\Gamma_2(d)\xrightarrow{r_2}Sp_4(\mathbb{Z}/2\mathbb{Z})\to 1$$
\end{Lemma}
\begin{proof}
	Clearly $\Gamma_2(2d)$ is a subgroup of $\Gamma_2(d)$ and $i(\Gamma_2(2d))=\Ker(r_2)$. Therefore, it is enough to prove that $r_2$ is surjective. To do this we use the following formula \cite[Page 222]{Igu}:
	$$|\Gamma_2:\Gamma_2(h)| = h^{10}\prod_{p|h,p\neq 1}\prod_{1\leq k\leq 2}(1-p^{-2k}).$$
	We have
	$$
	\begin{array}{lcl}
	|\Gamma_2:\Gamma_2(d)| & = & d^{10}\prod_{p|d,p\neq 1}(1-p^{-2})(1-p^{-4}), \\ 
	|\Gamma_2:\Gamma_2(2d)| & = & (2d)^{10}\prod_{p|2d,p\neq 1}(1-p^{-2})(1-p^{-4})=\\
	&   & 2^{10}d^{10}(1-2^{-2})(1-2^{-4})\prod_{p|d,p\neq 1}(1-p^{-2})(1-p^{-4}).
	\end{array} 
	$$
	Therefore
	$$\frac{|\Gamma_2:\Gamma_2(2d)|}{|\Gamma_2:\Gamma_2(d)|} = 2^{10}(1-2^{-2})(1-2^{-4}) = 720.$$
	Finally, since $|Sp_4(\mathbb{Z}/2\mathbb{Z})| = 720$ we conclude that $r_2$ is surjective.
\end{proof}
Let $D\in M_{g}(\mathbb{Z})$ be a diagonal $g\times g$ matrix. We define the subgroup $\Gamma_D\subset M_{2g}(\mathbb{Z})$ as:
\begin{equation}\label{gammad}
\Gamma_D := \left\lbrace R\in M_{2g}(\mathbb{Z}) \: | \: R
\left(\begin{matrix}
0 & D \\ 
-D & 0
\end{matrix}\right) R^t = 
\left(\begin{matrix}
0 & D \\ 
-D & 0
\end{matrix}\right)
\right\rbrace,
\end{equation}
and the subgroup $\Gamma_D(D)\subset\Gamma_D$ as:
\begin{equation}\label{gammadd}
\Gamma_D(D):= \left\lbrace \left(\begin{matrix}
\mathbf{A} & \mathbf{B} \\ 
\mathbf{C} & \mathbf{D}
\end{matrix}\right)\in \Gamma_D \: | \: \mathbf{A}-I \equiv \mathbf{B} \equiv \mathbf{C} \equiv \mathbf{D}-I \equiv 0\ \mod (D)
\right\rbrace,
\end{equation}
where $\mathbf{M} \equiv 0\ \mod (D)$ if and only if $\mathbf{M}\in D\cdot M_g(\mathbb{Z})$. See \cite[Sect.8.3.1]{BL} for details on this group. 
\begin{Lemma}\label{l2}
Let $D = \diag(d_1,d_2)$, where $d_1,d_2$ are odd integers. Then the reduction modulo $2$ morphism $r_2: \Gamma_D\rightarrow GL_{4}(\mathbb{Z}/2\mathbb{Z})$
is surjective onto $Sp_{4}(\mathbb{Z}/2\mathbb{Z})$. The restriction morphism $r_2: \Gamma_D(D)\rightarrow Sp_{4}(\mathbb{Z}/2\mathbb{Z})$
is also surjective, and we have the following exact sequence
\begin{equation}\label{level-sequence}
1\to \Gamma_D(2D)\rightarrow \Gamma_D(D)\xrightarrow{r_2}Sp_4(\mathbb{Z}/2\mathbb{Z})\to 1.
\end{equation}
\end{Lemma}
\begin{proof}
Since $d_1,d_2$ are odd we have 
$$\left(\begin{matrix}
0 & D \\ 
-D & 0
\end{matrix}\right)\equiv  
\left(\begin{matrix}
0 & I \\ 
-I & 0
\end{matrix}\right)
\mod(2).$$
Furthermore, if $R\in \Gamma_D$ the equality 
$R\left(\begin{matrix}
0 & D \\ 
-D & 0
\end{matrix}\right) R^t = 
\left(\begin{matrix}
0 & D \\ 
-D & 0
\end{matrix}\right)$
yields
$r_2(R)\left(\begin{matrix}
0 & I \\ 
-I & 0
\end{matrix}\right) r_2(R)^t = 
\left(\begin{matrix}
0 & I \\ 
-I & 0
\end{matrix}\right)$, hence $r_2$ is well defined. Let $d = d_1d_2$. By Lemma \ref{l1} the reduction modulo two $\Gamma_2(d)\rightarrow Sp_4(\mathbb{Z}/2\mathbb{Z})$ is surjective. Now, let  us take a symplectic matrix $M\in Sp_4(\mathbb{Z}/2\mathbb{Z})$, and let $N\in \Gamma_2(d)\subset Sp_4(\mathbb{Z})$ be a symplectic matrix such that $N\equiv M  \mod(2) $. Since $N\in \Gamma_2(d)$ we may write  
$$N = \left(\begin{matrix}
dx_{11}+1 & dx_{12} & dy_{11} & dy_{12}\\ 
dx_{21} & dx_{22}+1 & dy_{21} & dy_{22}\\
dz_{11} & dz_{12} & dw_{11}+1 & dw_{12}\\
dz_{21} & dz_{22} & dw_{21} & dw_{22}+1
\end{matrix}\right)
$$
Let us consider the group
$$\Gamma_D(\mathbb{Q}) := \left\lbrace R\in M_{4}(\mathbb{Q}) \: | \: R
\left(\begin{matrix}
0 & D \\ 
-D & 0
\end{matrix}\right) R^t = 
\left(\begin{matrix}
0 & D \\ 
-D & 0
\end{matrix}\right)
\right\rbrace,$$
the rational analogue of $\Gamma_D$. Then we have an isomorphism
$f_D:\Gamma_D(\mathbb{Q})\rightarrow Sp_4(\mathbb{Q}),$
defined by
$f_D(R) = \left(\begin{matrix}
I & 0 \\ 
0 & D^{-1}
\end{matrix}\right)
R
\left(\begin{matrix}
I & 0 \\ 
0 & D
\end{matrix}\right)
$.
Therefore the matrix 
$$R = f_D^{-1}(N) = 
\left(\begin{matrix}
d_1d_2x_{11}+1 & d_1d_2x_{12} & d_2y_{11} & d_1y_{12}\\ 
d_1d_2x_{21} & d_1d_2x_{22}+1 & d_2y_{21} & d_1y_{22}\\
d_1^2d_2z_{11} & d_1^2d_2z_{12} & d_1d_2w_{11}+1 & d_1^2w_{12}\\
d_1d_2^2z_{21} & d_1d_2^2z_{22} & d_2^2w_{21} & d_1d_2w_{22}+1
\end{matrix}\right)
$$
is in $\Gamma_D = \Gamma_D(\mathbb{Z})$ because $N$ is a matrix with integer entries. It is easy to see that 
$R\equiv N\equiv M  \mod(2),$ and the reduction modulo two $r_2: \Gamma_D\rightarrow Sp_4(\mathbb{Z}/2\mathbb{Z})$ is surjective.\\
Now, by Lemma \ref{l1} the reduction modulo two $\Gamma_2(d^2)\rightarrow Sp_4(\mathbb{Z}/2\mathbb{Z})$ is surjective since $d^2$ is an odd integer. We proceed as before. Let $M\in Sp_4(\mathbb{Z}/2\mathbb{Z})$ be a symplectic matrix, and let $N^{'}\in \Gamma_2(d^2)\subset Sp_4(\mathbb{Z})$ be a symplectic matrix such that $N^{'}\equiv M  \mod(2) $. Since $N^{'}\in \Gamma_2(d^2)$ we may write  
$$N^{'} = \left(\begin{matrix}
d^2x_{11}+1 & d^2x_{12} & d^2y_{11} & d^2y_{12}\\ 
d^2x_{21} & d^2x_{22}+1 & d^2y_{21} & d^2y_{22}\\
d^2z_{11} & d^2z_{12} & d^2w_{11}+1 & d^2w_{12}\\
d^2z_{21} & d^2z_{22} & d^2w_{21} & d^2w_{22}+1
\end{matrix}\right)
$$
Therefore 
$$R^{'} = f_D^{-1}(N^{'}) = 
\left(\begin{matrix}
d_1^{2}d_2^{2}x_{11}+1 & d_1^2d_2^2x_{12} & d_1d_2^2y_{11} & d_1^2d_2y_{12}\\ 
d_1^2d_2^2x_{21} & d_1^2d_2^2x_{22}+1 & d_1d_2^2y_{21} & d_1^2d_2y_{22}\\
d_1^3d_2^2z_{11} & d_1^3d_2^2z_{12} & d_1^2d_2^2w_{11}+1 & d_1^2d_2w_{12}\\
d_1^2d_2^3z_{21} & d_1^2d_2^3z_{22} & d_1d_2^3w_{21} & d_1^2d_2^2w_{22}+1
\end{matrix}\right)
$$
is in $\Gamma_D = \Gamma_D(\mathbb{Z})$. In particular, simply by checking the definition, one sees that $R^{'}$ is actually in $\Gamma_D(D)$. Hence we conclude that $r_2: \Gamma_D(D)\rightarrow Sp_{4}(\mathbb{Z}/2\mathbb{Z})$ is surjective. Now, let us consider the group
$$\Gamma_D(2D):= \left\lbrace \left(\begin{matrix}
\textbf{A} & \textbf{B}\\ 
\textbf{C} & \textbf{D}
\end{matrix}\right)\in \Gamma_D \: | \: \textbf{A}-I \equiv \textbf{B} \equiv \textbf{C} \equiv \textbf{D}-I \equiv 0 \mod(2D)
\right\rbrace.$$ 
Clearly, $\Gamma_D(2D)$ is a subgroup of $\Gamma_D(D)$. A matrix 
$M = \left(\begin{matrix}
\textbf{A} & \textbf{B} \\ 
\textbf{C} & \textbf{D}
\end{matrix}\right) \in \Gamma_D(D)$ 
lies in $\Ker(r_2)$ if and only if $\textbf{A}\equiv \textbf{D}\equiv I \mod(2)$, and $\textbf{B}\equiv \textbf{C} \equiv 0 \mod(2)$. Therefore, since $M\in \Gamma_D(D)$, we see that $M\in\Ker(r_2)$ if and only if $\textbf{A}\equiv \textbf{D}\equiv I \mod(2D)$, and $\textbf{B}\equiv \textbf{C} \equiv 0 \mod(2D)$, that is $M\in\Gamma_D(2D)$. We conclude that $\Ker(r_2) = \Gamma_D(2D)$. Hence we get the exact sequence in the statement.   
\end{proof} 
\subsubsection{The moduli spaces $\mathcal{A}_2(d_1,d_2)^{-}_{sym}$ and $\mathcal{A}_2(d_1,d_2)^{+}_{sym}$}
Let us consider the Siegel upper half-space $\mathbb{H}_2$. As before, let $D = \diag(d_1,d_2)$ with $d_1,d_2$ odd. By \cite[Section 8.2]{BL} and the Baily-Borel theorem \cite{BB}, since $\Gamma_D$ is an arithmetic congruence subgroup, the quasi-projective variety $\mathcal{A}_D = \mathbb{H}_g/\Gamma_D$ is the moduli space of abelian varieties with a polarization of of type $D$: see also \cite[Proposition 1.21]{HKW}. Furthermore, by \cite[Section 8.3]{BL} and \cite{BB}, the quasi-projective variety $\mathcal{A}_D(D) = \mathbb{H}_g/\Gamma_D(D)$  is the moduli space of polarized abelian varieties of type $D$ with level $D$ structure. Since by Lemma \ref{sequence-split} a level structure is equivalent to a symmetric theta structure, we are now going to investigate the action of these arithmetic subgroups on the set $\vartheta(A)$ of the $16$ theta characteristics (equivalently the set of symmetric line bundles). Each of them admits a unique symmetric theta structure.\\
The set of symmetric theta divisors is in bijection with the set of half-integer characteristics (\cite[Section 2]{Igu} or \cite[Sections 4.6 and 4.7]{BL}) and the action of $\Gamma_2$ on $\mathbb{H}_2$ induces an action on characteristics given by the formula (\ref{char-action}).
\begin{Lemma}\cite[Section 2]{Igu}
The action of $\Gamma_2$ on half-integer characteristics defined by formula (\ref{char-action}) has two orbits distinguished by the invariant
$$\mathbf{e}(m) = (-1)^{4ab^t}\in \{\pm 1\}.$$
\end{Lemma}
We say that $m=(a,b) \in \frac{1}{2} \Z^4/\Z^4$ is an even (respectively odd) half-integer characteristic if $\mathbf{e}(m) = 1$ (respectively $\mathbf{e}(m) = -1$).
Since $A[2]$ is a $\Z/2\Z$-vector space of dimension $4$, $\Gamma_2=Sp_4(\Z)$ operates on the set of theta characteristics through reduction modulo $2$, hence via $Sp_4(\Z/2\Z)$. Now, recall from Lemma \ref{l2} the exact sequence \ref{level-sequence}. 
Let $O^{-}_4(\mathbb{Z}/2\mathbb{Z})\subset Sp_4(\mathbb{Z}/2\mathbb{Z})$ be the stabilizer of an odd quadratic form. We have an isomorphism $Sp_4(\mathbb{Z}/2\mathbb{Z})\cong S_6$, where $S_6$ is the symmetric group, under which $Sp_4(\mathbb{Z}/2\mathbb{Z})$ acts on the set of odd quadratic forms by permutations. As a consequence, for the stabilizer subgroup of an odd theta characteristic we also have $O^{-}_4(\mathbb{Z}/2\mathbb{Z})\cong S_5\subset S_6$.  
\begin{Definition}
We denote by $\Gamma_2(d_1,d_2)^{-}$ the group $\Gamma_2(d_1,d_2)^{-}:= r_2^{-1}(O^{-}_4(\mathbb{Z}/2\mathbb{Z}))\subset\Gamma_D(D)$ that fits  in the exact sequence
$$1\to \Gamma_D(2D)\rightarrow \Gamma_2(d_1,d_2)^{-}\xrightarrow{r_2}O^{-}_4(\mathbb{Z}/2\mathbb{Z})\to 1$$
\end{Definition}
Explicitly, we can write
$\Gamma_2(d_1,d_2)^{-}= \left\lbrace Z\in \Gamma_D(D) \: | \: Z\mod (2) \equiv \Sigma, \: \Sigma\in O^{-}_4(\mathbb{Z}/2\mathbb{Z})  \right\rbrace.$
Therefore, we have $\Gamma_D(2D)\subset\Gamma_2(d_1,d_2)^{-}\subset\Gamma_D(D)$. Furthermore, $|\Gamma_D(D):\Gamma_D(2D)| = 6!$ and $|\Gamma_2(d_1,d_2)^{-}:\Gamma_D(2D)| = 5!$ imply that
$|\Gamma_D(D):\Gamma_2(d_1,d_2)^{-}| = 6.$ Since $\Gamma_2(d_1,d_2)^{-}$ is an arithmetic congruence subgroup, thanks to the Baily-Borel theorem \cite{BB}, we have that the quotient 
$$\mathcal{A}_2(d_1,d_2)^{-}_{sym} := \mathbb{H}_2/\Gamma_2(d_1,d_2)^{-}$$
is a quasi-projective variety. The variety $\mathcal{A}_2(d_1,d_2)^{-}_{sym}$ is the moduli space of polarized abelian surfaces $(A,H)$ with level $(d_1,d_2)$ structure, a symmetric theta structure and an odd line bundle in $\pic^H(A)$. The morphism
$f^{-}:\mathcal{A}_2(d_1,d_2)^{-}_{sym}\rightarrow\mathcal{A}_D(D)$
that forgets the choice of the odd line bundle is of degree $|\Gamma_D(D):\Gamma_2(d_1,d_2)^{-}| = 6$.\\

Let $O^{+}_4(\mathbb{Z}/2\mathbb{Z})\subset Sp_4(\mathbb{Z}/2\mathbb{Z})$ be the stabilizer of an even quadratic form. The proofs in the even case are very similar to the odd case.
\begin{Definition}
We denote by $\Gamma_2(d_1,d_2)^{+}$ the group $\Gamma_2(d_1,d_2)^{+}:= r_2^{-1}(O^{+}_4(\mathbb{Z}/2\mathbb{Z}))\subset\Gamma_D(D)$ that fits  in the exact sequence
$$1\to \Gamma_D(2D)\rightarrow \Gamma_2(d_1,d_2)^{+}\xrightarrow{r_2}O^{+}_4(\mathbb{Z}/2\mathbb{Z})\to 1$$
\end{Definition}
The stabilizer $O^{+}_4(\mathbb{Z}/2\mathbb{Z})\subset Sp_4(\mathbb{Z}/2\mathbb{Z})$ of an even quadratic form has order $|O^{+}_4(\mathbb{Z}/2\mathbb{Z})| = 72$ and $|\Gamma_D(D):\Gamma_2(d_1,d_2)^{+}| = 10.$
Using again the Baily-Borel theorem \cite{BB}, we get that the quotient 
$$\mathcal{A}_2(d_1,d_2)^{+}_{sym} := \mathbb{H}_2/\Gamma_2(d_1,d_2)^{+}$$
is a quasi-projective variety. By construction, it is the moduli space of polarized abelian surfaces $(A,H)$ with level $(d_1,d_2)$ structure and an even theta characteristic. The morphism $f^{+}:\mathcal{A}_2(d_1,d_2)^{+}_{sym}\rightarrow\mathcal{A}_D(D)$ forgetting the even theta characteristic has degree $|\Gamma_D(D):\Gamma_2(d_1,d_2)^{+}| = 10$.
\begin{Remark}
A particular case of $\Gamma_2(d_1,d_2)^{+}$ in the case of $(3,3)$-level structure is the group $\Gamma_2(3,6)$ studied by G. van der Geer in \cite{vdG}. In that case the moduli space $\A_2(3,3)^+_{sym}$ turned out to be a degree $10$ cover of the Burkhardt quartic hypersurface in $\P^4$. The moduli space $\A_2(3,3)^-_{sym}$ was proven to be rational in \cite{Bo}.
\end{Remark}
\begin{Remark}
The results in this section hold in greater generality for any $g$. In particular, arguing as in the proof of Lemma \ref{l2}, if $D = \diag(d_1,...,d_g)$, where the $d_i$'s are odd integers, we have the same exact sequence (\ref{level-sequence}) with $2g$ instead of $4$. The other  definitions are completely analogous.
\end{Remark}
\subsubsection{The even case}
In this section we will quickly investigate, since in the end our focus is on $(1,d)$ polarizations, the case where $d_1$ and $d_2$ are both even. Contrary to the odd case, if $D = \diag(d_1,...,d_g)$ and $d_i$ is even for some $i$, the reduction modulo two of a matrix in $\Gamma_D$ is not necessarily an element of $Sp_{2g}(\mathbb{Z}/2\mathbb{Z})$. The following elementary example shows one instance of this phenomenon.
\begin{Example}\label{nonsurj}
For instance, if $g = 2$ and $d_1 = 1, d_2 = 2$ the matrix $M = \left(\begin{matrix}
1 & 1 & 1 & 1\\ 
2 & 1 & 2 & 1\\
2 & 0 & 1 & 1\\
0 & 2 & 2 & 1
\end{matrix}\right)$ is in $\Gamma_D$. However, if we denote by $N$ its reduction modulo two we have 
$N\cdot\left(\begin{matrix}
0 & 0 & 1 & 0\\ 
0 & 0 & 0 & 1\\
1 & 0 & 0 & 0\\
0 & 1 & 0 & 0
\end{matrix}\right)\cdot
N^t = \left(\begin{matrix}
0 & 0 & 0 & 1\\ 
0 & 0 & 1 & 1\\
0 & 1 & 0 & 0\\
1 & 1 & 0 & 0
\end{matrix}\right)$.
Hence $N\notin Sp_{2}(\mathbb{Z}/2\mathbb{Z})$.
\end{Example}
Hence the action on theta-characteristics is not well defined, and we will shortly find a modular reason for this. In fact we will see that in this case we do not need to keep track of the symmetric line bundles representing the polarization.\\
Following Theorem \ref{symmetry}, we observe that the situation in the even case is somehow opposite to the odd one. In fact, there exists only one symmetric line bundle representing the polarization that admits a symmetric theta structure. On the other hand, if we fix a level structure, there are $16$ symmetric theta structures that induce that level structure, corresponding to the elements of $A[2]\cap K(L)=A[2]$ (see Remark \ref{remsymmetry}). The prototypical example of such a moduli space is $\A_2(2,4)$, the moduli space of abelian surfaces with a $(2,2)$-polarization with level structure, plus a symmetric theta structure. It is well known, see for instance \cite{DO}, that the Satake compactification of $\A_2(2,4)$ is isomorphic to $\P^3$. 
More generally, when $d_1 \neq d_2$ are even, one can naturally generalize the definition of $\Gamma_2(2,4)$ and define an arithmetic subgroup, which we denote by $\Gamma_2(D,2D)$, as follows:
\begin{equation}\label{d1d2even}
	\Gamma_2(D,2D):=\left\lbrace\left(\begin{array}{cc} \mathbf{A} & \mathbf{B} \\ \mathbf{C} & \mathbf{D} \\ \end{array}\right) \in \Gamma_D(D)\: | \: \diag(\mathbf{B})\equiv\ \diag(\mathbf{C}) \equiv 0\ \mod(2D) \right\rbrace.
\end{equation}
Note that this consists of $\Gamma_2(d_1,d_2)\cap\Gamma_2(1,2)$.
The fact that the quotient of $\mathbb{H}_2$ via this group parametrizes $(d_1,d_2)$-polarized abelian surfaces with a symmetric theta structure is equivalent to the fact that in this case, as explained in Theorem \ref{symmetry}, symmetric theta structures correspond to points of $A[2]$, that are in (non- canonical) bijection with half-integer characteristics. In fact, thanks to the action on characteristics of equation (\ref{char-action}), we see that $\Gamma_2(D,2D)$ is the stabilizer inside $\Gamma_2(d_1,d_2)$ of the zero characteristic. On the other hand, in this case the action of the corresponding level group $\Gamma_2(d_1,d_2)$ on the set of characteristics is transitive, as it operates through the quotient $\Gamma_2(d_1,d_2)/ \Gamma_2(D,2D)\cong (\Z/2\Z)^4$. 
Since $\Gamma_2(D,2D)$ is an arithmetic congruence subgroup, thanks to the Baily-Borel theorem \cite{BB}, we have that the quotient 
$$\mathcal{A}_2(d_1,d_2)_{sym} := \mathbb{H}_2/\Gamma_2(D,2D)$$
is a quasi-projective variety. By construction it is the moduli space of polarized abelian surfaces with level $(d_1,d_2)$ structure and a symmetric theta structure. The following Lemma is straightforward.
\begin{Lemma}
The $16$ different symmetric theta-structures that induce a given level structure in the even case are a principal homogeneous space under the action of $(\Z/2\Z)^4$, embedded in the centralizer subgroup $C_{\imath^D}\subset \Aut_{\mathbb{C}^{*}}(\mathcal{H}(D))$ via the first arrow of the exact sequence (\ref{sequenza}).
\end{Lemma}
In fact the exact sequence (\ref{autsym-sequence}) reduces to the sequence 
$$1\to A[2] \to C_{\imath^D} \to Sp(D) \to 1$$
when both the coefficients $d_i$ are even, and we have seen that the 2-torsion points correspond to the symmetric theta structures compatible with a given level structure. 

\begin{Remark}
It is straightforward to check that there exists 
 a forgetful map
$\mathcal{A}_2(d_1,d_2)_{sym}\to \mathcal{A}_2(d_1,d_2),$
forgetting the theta structure, which has degree $16=\#(\Z/2\Z)^4$.
\end{Remark}
\subsubsection{The intermediate case.}\label{intcase}
Let us now come to the intermediate type. By this we mean polarizations where $d_1$ is odd and $d_2$ is even. Following as usual Theorem \ref{symmetry}, we have four symmetric line bundles inside the equivalence class of the polarization that admits a symmetric theta structure. Each of them admits four symmetric theta-structures that induce a given level structure. In fact (see Remark \ref{remsymmetry}) the $4$ symmetric line bundles correspond to elements of the quotient
\begin{equation}\label{restrizia}
A[2]/(K(L)\cap A[2])\cong\Z/2\Z \times \Z/2\Z.
\end{equation}
\begin{Lemma}\label{parity}
Among the symmetric line bundles of the set (\ref{restrizia}), there are $3$ inducing an even quadratic form and $1$ an even form.
\end{Lemma}
\begin{proof}
In order to show this it is enough to consider and abelian surface $A=E_1 \times E_2$, with $E_1$ an odd $d_1$-polarized elliptic curve and $E_2$ a second elliptic curve with an even $d_2$-polarization. This is a specialization of the general case, and clearly the quotient mods out the $2$-torsion points of the second elliptic curve and the claim follows.
\end{proof}
On the other hand (see Remark \ref{remsymmetry}) the $4$ symmetric theta-structures inducing a given level structure correspond to the points of $(K(L)\cap A[2])$. It is easy to see that this subgroup is isomorphic once again to $\Z/2\Z \times \Z/2\Z.$
Our goal is then to construct moduli spaces for the datum of a symmetric line bundle representing the polarization plus the choice of a compatible symmetric theta structure. Of course, because of Lemma \ref{parity}, we will need to consider two different moduli spaces according to the parity of the theta characteristic. 

Note that the rank two subgroup $K(L)\cap A[2]\subset A[2]$ induces a decomposition of $A[2]$ as $(K(L)\cap A[2] )\times(A[2]/(K(L)\cap A[2]))$. Of course both groups are isomorphic to $(\Z/2\Z)^2$, and in the construction of the arithmetic group we will want to distinguish the action of the group on each one. 
The action of the group will basically imitate the odd case on $A[2]/(K(L)\cap A[2])$ and the even case on $K(L)\cap A[2]$. The reason is once again the exact sequence of Lemma \ref{symsequenza}. Here one copy of $(\Z/2\Z)^2\subset C_{\imath^D}$ comes from $A[2]\cap K(L)$ and operates transitively on the 4 symmetric theta structures. The second copy of $(\Z/2\Z)^2$ lifts up from $Sp(D)$ and it operates on the four symmetric line bundles (admitting a symmetric theta structure) preserving the parity. More concretely, we want to construct two subgroups (distinguished by the parity of the theta characteristics) of $\Gamma_D(D)$ with the following features:  
\begin{itemize}
\item[(a)] \rm since the four theta characteristics of Lemma \ref{parity} are in bijection with elements of $(\Z/2\Z)^2$, the action of each subgroup on $\Lambda$, reduced modulo 2, must descend to the action of $O^\pm_2(\Z/2\Z)$ (depending on the parity of the theta characteristic that we want) on $A[2]/(K(L)\cap A[2])$. This implies that the corresponding moduli spaces parametrize one (even or odd) symmetric line bundle among the four.
\item[(b)] On the other hand, by imitating the action of the arithmetic group of the even case, we need both our subgroups to operate as $\Gamma_1(d_2,2d_2)$ when acting on $K(L)\cap A[2]$. This in turn implies that the quotient by our arithmetic subgroups will also keep track of the four symmetric theta structures.
\end{itemize}
There exists only one odd symmetric line bundle representing a polarization of intermediate type that admits a symmetric theta structure. On the group theoretical side this is equivalent to the fact that the subgroup $O^-_2(\Z/2\Z)$ is isomorphic to $Sp_2(\Z/2\Z)$. Things are a little more complicated in the even case, since in that case we really want the induced action on $A[2]/(K(L)\cap A[2])$ to factor through $O^+_2(\Z/2\Z)$ which is a proper subgroup of index $3$ of $Sp_2(\Z/2\Z)$, as it is explained in the following remark.
\begin{Remark}\label{groups}
Let us outline briefly the relations between $O^{\pm}_2(\Z/2\Z)$ and $Sp_2(\Z/2\Z)$. From \cite[Proposition 2.9.1]{KL} we see that $O^{-}_2(\Z/2\Z)\cong D_{6}$ is the dihedral group of order six, and $O^{+}_2(\Z/2\Z)$ is cyclic of order two. In particular $|Sp_2(\mathbb{Z}/2\mathbb{Z}):O^{-}_2(\mathbb{Z}/2\mathbb{Z})|=1$ and $|Sp_2(\mathbb{Z}/2\mathbb{Z}):O^{+}_2(\mathbb{Z}/2\mathbb{Z})|=3$.
\end{Remark}
Following (a) and (b) above, we define two arithmetic groups, for odd $d_1$ and even $d_2$:
$$
\begin{array}{l}
\Gamma_2(d_1,d_2)^+_{sym}:=\{ N\in \Gamma_D(D)\: | \: N_{|A[2]/(K(L)\cap A[2])} \in O^+(2,\Z/2\Z),\ N_{|K(L)\cap A[2]} \in \Gamma_1(d_2,2d_2)\},\\ 
\Gamma_2(d_1,d_2)^-_{sym}:=\{ N\in \Gamma_D(D)\: | \: N_{|K(L)\cap A[2]} \in \Gamma_1(d_2,2d_2)\}.
\end{array} 
$$
Moreover, since $d_1|d_2$, in this case $d_2$ must be an even multiple of $d_1$. By the Baily-Borel theorem \cite{BB}, and since $\Gamma_2(d_1,d_2)^\pm_{sym}$ are arithmetic congruence subgroups, we get two quasi-projective varieties
$$
\begin{array}{l}
\A_2(d_1,d_2)^+_{sym}  :=  \mathbb{H}_2/\Gamma_2(d_1,d_2)^+_{sym}, \\ 
\A_2(d_1,d_2)^-_{sym}  :=  \mathbb{H}_2/\Gamma_2(d_1,d_2)^-_{sym},
\end{array} 
$$
parametrizing abelian surfaces with a polarization of type $(d_1,d_2)$, a symmetric theta structure and an even (respectively odd) theta characteristic.
\begin{Remark}
By Proposition \ref{actions}, it is straightforward to see that $\A_2(d_1,d_2)^+_{sym}$ (respectively $\A_2(d_1,d_2)^-_{sym}$) is a $12$ to $1$ (respectively $4$ to $1$) cover of the moduli space of polarized abelian surfaces with a level structure $\A_2(d_1,d_2)^{lev}$.
\end{Remark}
\section{Moduli of $(1,d)$-polarized surfaces, with symmetric theta structure and a theta characteristic: the theta-null map}\label{seccinque}
In the rest of the paper we are going to study the birational geometry of some moduli spaces of abelian surfaces with a level $(1,d)$-structure, a symmetric theta structure and  an odd theta characteristic, which will encode the choice of a symmetric line bundle representing the polarization. The study of abelian surfaces with an even theta characteristic will be the object of further work \cite{BM}. Our main tool will be theta functions, more precisely theta constants mapping to the projective space. Before we start a case-by-case analysis, let us make a useful observation that holds for any polarization type $(d_1,d_2)$. The following remark is due to an anonymous referee.
\begin{Remark}\label{level-image}
Two  abelian surfaces with the same level structure have different images inside $\P^{d_1d_2-1}$ (in fact they can be identified only when they are endowed with a theta structure), but their intersections with the projective eigenspaces $\P(H^0(A,L)^\pm)^*$ are two (possibly empty) finite sets determined uniquely by the level structure because they are exactly the base points of the linear systems $\P(H^0(A,L)^\mp)^*$. See Proposition \ref{dim-e-bl} for more details.
\end{Remark}
\subsection{The odd case}\label{nullodd}
When $d_1$ and $d_2$ are odd, the general construction of the map from the moduli spaces $\A_2(d_1,d_2)^\pm_{sym}$ is the following. We start from the datum $(A, H,L,\psi)$ of an abelian surface with a $(d_1,d_2)$-polarization $H$, a level structure $\psi$ and $L\in \pic^H(A)$ symmetric (in fact the datum of $H$ is redundant and we will omit it in the following). As we have seen, there exists $16$ symmetric line bundles, $10$ even and $6$ odd, representing the polarization. On the other hand, thanks to Lemma \ref{sequence-split} we know that there is only one symmetric theta structure $\Psi$ that induces $\psi$. Let us denote it by $\Psi$. This means that we can take canonical bases for the eigenspaces of the space of delta functions $V_2(D)$ with respect to the action of the involution $\imath_D$ defined in Section \ref{secdue}. From Section \ref{sectre} we recall that the eigenspaces of the projective space $\P(V_2(D))$ of delta functions are respectively $\P_+^{\frac{d_1d_2-1}{2}}$ and $\P_-^{\frac{d_1d_2-3}{2}}$.\\
If $L$ is even (respectively odd), the symmetric theta structure gives an identification of $\P(H^0(A,L)^+)^*$ with $\P_+^{\frac{d_1d_2-1}{2}}$ (respectively $\P_-^{\frac{d_1d_2-3}{2}}$). Similarly, we identify $\P(H^0(A,L)^-)^*$ with $\P_-^{\frac{d_1d_2-3}{2}}$ (respectively $\P_+^{\frac{d_1d_2-1}{2}}$) if $L$ is even (respectively odd).\\ 
Let $(A,\psi)\in\A_2(d_1,d_2)^{lev}$ be a polarized abelian surface with level structure. Then, recalling Proposition \ref{dim-e-bl}, we have that $A\cap \P_-^{\frac{d_1d_2-3}{2}}= A[2]^+$ if $L$ is odd, and it equals $A[2]^-$ if $L$ is even. On the other hand $A\cap \P_+^{\frac{d_1d_2-1}{2}}= A[2]^+$ if $L$ is even, and $A[2]^-$ if $L$ is odd. As we have pointed out in Remark \ref{level-image}, the sets $A[2]^+$ and $A[2]^-$ are uniquely determined by the level structure. Recall that the origin $0$ belongs to $A[2]^+$, and in fact the different choices of $L$ among the even (respectively odd) symmetric line bundles make the origin move along the intersection $A\cap\P_+^{\frac{d_1d_2-1}{2}}$ (respectively $\P_-^{\frac{d_1d_2-3}{2}}$), which in fact is made up of $10$ (respectively $6$) points. Hence finally we can define two maps 
\begin{equation}\label{theta+}
\begin{array}{cccc}
Th^+_{(d_1,d_2)}: & \A_2(d_1,d_2)^+_{sym} & \rightarrow & \P_+^{\frac{d_1d_2-1}{2}}\\ 
 & (A,L,\psi) & \mapsto & \Psi^+(\Theta_{d_1,d_2}(0))
\end{array} 
\end{equation}
and
\begin{equation}\label{theta-}
\begin{array}{cccc}
Th^-_{(d_1,d_2)}: & \A_2(d_1,d_2)^-_{sym} & \rightarrow & \P_-^{\frac{d_1d_2-3}{2}}\\ 
 & (A,L,\psi) & \mapsto & \Psi^-(\Theta_{d_1,d_2}(0)).
\end{array} 
\end{equation}
Here $\Theta_{d_1,d_2}$ is the map to $\P(H^0(A,L))^*$ given by the global sections of the polarization (in fact all anti-invariant sections vanish at zero), and $\Psi^+$ (respectively $\Psi^-$) is the identification of $\P(H^0(A,L)^+)^*$ with $\P_+^{\frac{d_1d_2-1}{2}}$ (respectively with $\P_-^{\frac{d_1d_2-3}{2}}$) induced by the symmetric theta structure $\Psi$ corresponding to $\psi$ when $L$ is even (respectively odd). 
\subsection{The even case}
As we have said in Section \ref{secquattro}, when $d_1$ and $d_2$ are both even, the right moduli space to consider is slightly different. In fact, we will consider the moduli space of abelian surfaces with a polarization $H$ of even type $(d_1,d_2)$ and a symmetric theta structure. Therefore, the map is the following:
\begin{equation}\label{thetaeven}
\begin{array}{cccc}
Th_{(d_1,d_2)}: & \A_2(d_1,d_2)^{sym} & \rightarrow & \P_+^{\frac{d_1d_2}{2}+1}\\ 
 & (A,\Psi) & \mapsto & \Psi^+(\Theta_{d_1,d_2}(0)).
\end{array} 
\end{equation}
where $\Theta_{d_1,d_2}(0)$ is the image of the origin through the map induced by the unique symmetric line bundle $L$ in the equivalence class of the polarization, $\Psi$ is the symmetric theta structure that induces the identification $\Psi^+:\P(H^0(A,L)^+)^* \rightarrow\P_+^{\frac{d_1d_2}{2}+1}$.\\
Recalling Proposition \ref{dim-e-bl}, we have that $A\cap \P_+^{\frac{d_1d_2}{2}+1}= A[2]$. Moreover (see Section \ref{secquattro}), given a level structure $\psi$ there exist $16$ symmetric theta structures inducing $\psi$, and (Remark \ref{level-image}) the level structure completely defines the set $A\cap \P_+^{\frac{d_1d_2}{2}+1}$, in this case the full set $A[2]$. The different choices of symmetric theta structure make the origin move along the $16$ points of the intersection $A\cap \P_+^{\frac{d_1d_2}{2}+1}$. 
The subgroup $(\Z/2\Z)^4$ of the centralizer $C_{\imath^D}\subset \Aut_{\mathbb{C}^{*}}(\mathcal{H}(D))$ of the involution $\imath^D$ has a natural representation $\widetilde{G}^+$ on $\P_+^{\frac{d_1d_2}{2}+1}$ (see Lemma \ref{decogtilde}) and it operates transitively on the set of symmetric theta structures inducing $\psi$ via this projective representation. This action induces the $16:1$ forgetful map
$$\A_2(d_1,d_2)^{sym} \to \A_2(d_1,d_2)^{lev}.$$
\subsection{The intermediate case}\label{nullint}
Now we come to what we feel to be the most interesting case. In the intermediate case (see Section \ref{secquattro}), we have two theta-null maps:
\begin{equation}\label{theta+int}
\begin{array}{cccc}
Th^+_{(d_1,d_2)}: & \A_2(d_1,d_2)_{sym}^+ & \rightarrow & \P_+^{\frac{d_1d_2}{2}}\\ 
 & (A,L,\Psi) & \mapsto & \Psi^+(\Theta_{d_1,d_2}(0))
\end{array} 
\end{equation}
and
\begin{equation}\label{theta-int}
\begin{array}{cccc}
Th^-_{(d_1,d_2)}: & \A_2(d_1,d_2)_{sym}^- & \rightarrow & \P_-^{\frac{d_1d_2-3}{2}}\\ 
 & (A,\Psi) & \mapsto & \Psi^-(\Theta_{d_1,d_2}(0)).
\end{array} 
\end{equation}
Here $\Psi$ is a symmetric theta structure, $\Psi^{\pm}$ the identification of $\P(H^0(A,L)^+)^*$ with the $\pm 1$-eigenspace, $L$ an even or odd (in the odd case there is no choice, since there is only one) line bundle and $\Theta_{d_1,d_2}(0)$ the image of the origin via the map induced by $L$.
Recall from Remark \ref{level-image} that the intersection sets of $A$ with the eigenspaces depend only on the level structure. Thanks to Proposition \ref{dim-e-bl}, we have that $A\cap \P_-^{\frac{d_1d_2-3}{2}}= A[2]^+$ if $L$ is odd, and it equals $A[2]^-$ if $L$ is even. On the other hand $A\cap \P_+^{\frac{d_1d_2}{2}}= A[2]^+$ if $L$ is even, and $A[2]_-$ if $L$ is odd. The origin belongs to $A[2]^{+}$, and in fact the different choices of the $4$ symmetric theta structure and of the line bundle make the origin move along the intersection of $A$ with the eigenspaces. If $L$ is the unique odd line bundle only the action of $\Z/2\Z^2\subset C_{\imath^D}$ operates transitively on $A\cap \P_-^{\frac{d_1d_2-3}{2}}$, via the representation $\widetilde{G}^-$ (see Lemma \ref{decogtilde}) and induces the natural $4$ to $1$ forgetful map of the symmetric theta structure $\A_2(d_1,d_2)_{sym}^- \to \A_2(d_1,d_2)^{lev}$. On the other hand, if we concentrate on the even moduli space, then the cardinality of $A\cap \P_+^{\frac{d_1d_2}{2}}$ equals $12$ (see Proposition \ref{actions}) and this equals in fact $\#(\Z/2\Z)^2$ times the $3$ choices of even line bundles. The moduli map that forgets the even theta characteristic and the symmetric theta function is in fact the $12$ to $1$ map $\A_2(d_1,d_2)_{sym}^+ \rightarrow \A_2(d_1,d_2)^{lev}$.
\section{Moduli of $(1,d)$ polarized surfaces, with symmetric theta structure and a theta characteristic: birational geometry}\label{secsei}
In this section we study the birational geometry of some of the moduli spaces of polarized abelian surfaces introduced in Section \ref{secquattro}. 
\subsection{Polarizations of type $(1,n)$ with $n$ odd.}
First we need to recall from \cite[Section 6]{GP2} a few results about the Heisenberg action on the ideal of a $(1,2d+1)$-polarized abelian surface embedded in $\P(H^0(A,L))^{*}\cong\P^{2d}$. In fact, the group $\mathcal{H}_{1,2d+1}$ (see Definition \ref{finiteheis}) acts naturally on $H^0(\P^{2d},\OO_{\mathbb{P}^{2d}}(2))$ and it decomposes it into $d+1$ mutually isomorphic irreducible representations of $\hodd$. Gross and Popescu construct a $(d+1)\times (2d+1)$ matrix
\begin{equation}\label{matrice}
(R_d)_{ij}=x_{j+i}x_{j-i},\ \ \ 0\leq i \leq  d,\ 0\leq j \leq 2d,
\end{equation}
where the indices are modulo $2d$. Each row of $(R_d)_{ij}$ spans an irreducible sub-representation inside $H^0(\P^{2d},\OO(2))$, and this way we obtain the decomposition into $(d+1)$ irreducible sub-representations. 
\begin{Definition}\label{di}
We shall indicate by $D_i\subset \P^{d-1}_-$ the locus in $\P^{d-1}_-$ where the restriction of $R_d$ has rank $\leq 2i$.
\end{Definition}
Since $x_0=0$ and $x_i=-x_{-i}$ on $\P^{d-1}_-$, we can use coordinates $x_1,\dots,x_d$. By substituting these coordinates inside the matrix (\ref{matrice}), one sees  that the $j^{th}$ and the $(2d+1-j)^{th}$ column coincide on $\P^{d-1}_-$, if $j\neq 0$. In the same way we see that the leftmost $(d+1)\times (d+1)$ block of $R_d$ is anti-symmetric. Let us denote by $T_d$ the restriction of this block to $\P^{d-1}_-$. Hence $D_i$ is exactly the locus of $\P^{d-1}_-$ where $T_d$ is rank $\leq 2i$. The following result can be found in \cite[Lemma 6.3]{GP2}.
\begin{Lemma}\label{d1ed2}
For a general $\hodd$-invariant abelian surface $A\subset \P^{2d}$, $d\geq 3$, we have $A\cap \P^{d-1}_- \subset D_2$ and $A\cap \P^{d-1}_-\not\subset D_1$.
\end{Lemma}
\subsubsection{The case $n=7$} 
In order to analyze this case, we need to give a short introduction to varieties of sums of powers ($\VSP$ for short). These varieties parametrize decompositions of a general homogeneous polynomial $F\in k[x_{0},...,x_{n}]$ as sums of powers of linear forms. They have been widely studied from both the biregular \cite{IR}, \cite{Mu1}, \cite{Mu2}, \cite{RS} and the birational viewpoint \cite{MMe}, \cite{Ma}.\\
Let $\nu_{d}^{n}:\mathbb{P}^{n}\rightarrow\mathbb{P}^{N(n,d)}$, with $N(n,d) =\binom{n+d}{d}-1$ be the Veronese embedding induced by $\mathcal{O}_{\mathbb{P}^{n}}(d)$, and let $V_{d}^{n} = \nu_{d}^{n}(\mathbb{P}^{n})$ be the corresponding Veronese variety. Let $F\in k[x_0,...,x_n]_{d}$ be a general homogeneous polynomial of degree $d$.
\begin{Definition}\label{vspord}
Let $F\in\mathbb{P}^{N(n,d)}$ be a general point of $V_d^n$. Let $h$ be a positive integer and $\Hilb_h(\mathbb{P}^{n*})$ the Hilbert scheme of sets of $h$ points in $(\mathbb{P}^{n*})$. We define 
$$\VSP(F,h)^{o} := \{\{L_{1},...,L_{h}\}\in\Hilb_{h}(\mathbb{P}^{n*})\: | \: F\in \langle L_{1}^d,...,L_{h}^d\rangle\}\subseteq \Hilb_{h}(\mathbb{P}^{n*})\},$$
and $\VSP(F,h) := \overline{\VSP(F,h)^{o}}$ by taking the closure of $\VSP(F,h)^{o}$ in $\Hilb_{h}(\mathbb{P}^{n*})$. 
\end{Definition}
Suppose that the general polynomial $F\in\mathbb{P}^{N(n,d)}$ is contained in a $(h-1)$-linear space $h$-secant to $V_{d}^{n}$. Then, by \cite[Proposition 3.2]{Dol} the variety $\VSP(F,h)$ has dimension $h(n+1)-N(n,d)-1$. Furthermore, if $n = 1,2$ then for $F$ varying in an open Zariski subset of $\mathbb{P}^{N(n,d)}$ the variety $\VSP(F,h)$ is smooth and irreducible.

In order to apply this object to the study of abelian surfaces, we need to construct similar varieties parametrizing the decomposition of homogeneous polynomials as sums of powers of linear forms and admitting natural generically finite rational maps onto $\VSP(F,h)$.
\begin{Definition}
Let $F\in\mathbb{P}^{N(n,d)}$ be a general point. We define 
$$\VSP_{ord}(F,h)^{o} := \{(L_{1},...,L_{h})\in (\mathbb{P}^{n*})^{h} \: | \: 
F\in \langle L_{1}^d,...,L_{h}^d\rangle\}\subseteq (\mathbb{P}^{n*})^{h},$$
and $\VSP_{ord}(F,h) := \overline{\VSP_{ord}(F,h)^{o}}$ by taking the closure of $\VSP_{ord}(F,h)^{o}$ in $(\mathbb{P}^{n*})^{h}$. 
\end{Definition}
Note that $\VSP_{ord}(F,h)$ is a variety of dimension $h(n+1)-N(n,d)-1$. Furthermore, two general points of $\VSP_{ord}(F,h)$ define the same point of $\VSP(F,h)$ if and only if they differ by a permutation in the symmetric group $S_h$. Therefore, we have a generically finite rational map $\phi:\VSP_{ord}(F,h)\dasharrow \VSP(F,h)$ of degree $h!$ Now we consider the rational action of $S_{h-1}$ on $\VSP_{ord}(F,h)$ defined as follows:
$$
\begin{array}{cccc}
\rho: & S_{h-1}\times \VSP_{ord}(F,h) & \dasharrow & \VSP_{ord}(F,h)\\
 & (\sigma,(L_{1},...,L_{h})) & \longmapsto & (L_1,(\sigma(L_{2},...,L_{h})))
\end{array}
$$
\begin{Definition}\label{vsph}
We define the variety $\VSP_{h}(F,h)$ as the quotient 
$$\VSP_{h}(F,h) = \VSP_{ord}(F,h)/S_{h-1}$$
under the action of $S_{h-1}$ via $\rho$.
\end{Definition}
Note that $\VSP_{h}(F,h)$ admits a generically finite rational map $\psi:\VSP_{h}(F,h)\dasharrow \VSP(F,h)$ of degree $h$. By definition of the action $\rho$, the $h$ points on the fiber of $\psi$ over a general point $\{L_1,...,L_h\}\in \VSP(F,h)$ can be identified with the linear forms $L_1,...,L_h$ themselves. Furthermore we have the following commutative diagram of rational maps
  \[
  \begin{tikzpicture}[xscale=2.5,yscale=-1.0]
    \node (A0_0) at (0, 0) {$\VSP_{ord}(F,h)$};
    \node (A1_1) at (1, 1) {$\VSP_{h}(F,h)$};
    \node (A2_0) at (0, 2) {$\VSP(F,h)$};
    \path (A1_1) edge [->,dashed]node [auto] {$\scriptstyle{\psi}$} (A2_0);
    \path (A0_0) edge [->,dashed,swap]node [auto] {$\scriptstyle{\phi}$} (A2_0);
    \path (A0_0) edge [->,dashed]node [auto] {$\scriptstyle{\pi}$} (A1_1);
  \end{tikzpicture}
  \]
The variety $\VSP_{h}(F,h)$ can be explicitly constructed in the following way. Let us consider the incidence variety 
$$\mathcal{J}:=\{(l,\{L_1,...,L_h\})\: | \: l\in\{L_1,...,L_h\}\in \VSP(F,h)^{o}\}\subseteq \mathbb{P}^{n*}\times \VSP(F,h)^{o}.$$
Then $\VSP_{h}(F,h)$ is the closure $\overline{\mathcal{J}}$ of $\mathcal{J}$ in $\mathbb{P}^{n*}\times \VSP(F,h)$.
\begin{Remark}
In \cite{Mu1} Mukai proved that if $F\in k[x_0,x_1,x_2]_4$ is a general polynomial then $\VSP(F,6)$ is a smooth Fano $3$-fold $V_{22}$ of index $1$ and genus $12$. In this case we have a generically $6$ to $1$ rational map 
$$\psi:\VSP_6(F,6)\dasharrow \VSP(F,6).$$ 
By \cite{MS} and \cite[Corollary 5.6]{GP1}, under the same assumptions on $F$, the moduli space $\mathcal{A}_2(1,7)^{lev}$ of $(1,7)$-polarized abelian surfaces with canonical level structure is birational to $\VSP(F,6)$. Other interesting results on this moduli space are contained in \cite{Mar} and \cite{MR}. Our aim is now to give an interpretation of the covering $\VSP_6(F,6)$ in terms of moduli of $(1,7)$-polarized abelian surfaces with a symmetric theta structure and an odd theta characteristic.
\end{Remark}
Given an irreducible, reduced, non-degenerate variety $X\subset\P^N$ of dimension $n$, and a positive integer $h\leq N$ we denote by $\sec_h(X)$ the \emph{$h$-secant variety} of $X$. This is the subvariety of $\P^N$ obtained as the closure of the union of all $(h-1)$-planes $\langle x_1,...,x_{h}\rangle$ spanned by $h$ general points of $X$. The expected dimension of $\sec_h(X)$ is $\expdim(\sec_h(X)) = \min \{hn+h-1,N\}$. However, its actual dimension might be smaller. In this case $X$ is said to be $h$-defective, and the number $\delta_{h}(X) = nh+h-1 - \dim\sec_{h}(X)>0$ is called the $h$-secant defect of $X$.
We recall that a proper variety $X$ over an algebraically closed field is rationally connected if there is an irreducible rational curve through any two general points $x_1,x_2\in X$. Furthermore, rational connectedness is a birational property and indeed, if $X$ is rationally connected and $X\dasharrow Y$ is a dominant rational map, then $Y$ is rationally connected as well. By \cite[Corollary 1.3]{GHS}, if $f:X\rightarrow Y$ is a surjective morphism, where $Y$ and the general fiber of $\phi$ are rationally connected, then $X$ is rationally connected. 
\begin{Theorem}\label{vsp6rc}
The variety $\VSP_6(F,6)$ is rationally connected.
\end{Theorem}
\begin{proof}
Let us consider the Veronese variety $V^2_4\subset\mathbb{P}^{14} = \mathrm{Proj}(k[x_0,x_1,x_2]_4)$, and let $F\in\mathbb{P}^{14}$ be a homogeneous polynomial. If $F$ admits a decomposition as sum of powers of linear forms then its second partial derivatives have such a decomposition as well. Therefore, the second partial derivatives of $F$ are six points in $\mathbb{P}^5 = \proj(k[x_0,x_1,x_2]_{2})$ lying on a hyperplane. Hence the determinant of the $6\times 6$ \it catalecticant \rm matrix
$$M = \left(\begin{matrix}
\frac{\partial^2 F}{\partial x_0x_0} & \frac{\partial^2 F}{\partial x_0x_1} & \frac{\partial^2 F}{\partial x_0x_2} & \frac{\partial^2 F}{\partial x_1x_1} & \frac{\partial^2 F}{\partial x_1x_2} & \frac{\partial^2 F}{\partial x_2x_2} 
\end{matrix}\right)$$ 
is zero. It is well known that the secant variety $\sec_5(V_4^2)\subset\mathbb{P}^{14}$ is the irreducible hypersurface of degree $6$ defined by $\det(M) = 0$, see for instance \cite{LO}. Therefore $V_4^2$ is $5$-secant defective and $\delta_5(V_4^2) = 14-13 = 1$. Let us define the incidence variety
 \[
  \begin{tikzpicture}[xscale=1.5,yscale=-1.5]
    \node (A0_1) at (1, 0) {$\mathcal{X} = \{(\{L_1,...,L_5\},F)\: | \: F\in\left\langle L_1^4,...,L_5^4\right\rangle\}\subseteq \Hilb_{4}(\mathbb{P}^{2*})\times \sec_5(V_4^2)$};
    \node (A1_0) at (0, 1) {$\Hilb_{5}(\mathbb{P}^{2*})$};
    \node (A1_2) at (2, 1) {$\sec_5(V_4^2)\subset\mathbb{P}^{14}$};
    \path (A0_1) edge [->]node [auto] {$\scriptstyle{\psi}$} (A1_2);
    \path (A0_1) edge [->]node [auto,swap] {$\scriptstyle{\phi}$} (A1_0);
  \end{tikzpicture}.
  \]
The morphism $\phi$ is surjective and there exists an open subset $U\subseteq \Hilb_{5}(\mathbb{P}^{2*})$ such that for any $Z\in U$ the fiber $\phi^{-1}(Z)$ is isomorphic to $\mathbb{P}^{4}$, so $\dim(\phi^{-1}(Z)) = 4$. The morphism $\psi$ is dominant and for a general point $F\in\sec_5(V_4^2)$ we have 
$$\dim(\psi^{-1}(F)) = \dim(\mathcal{X})-\dim(\sec_5(V_4^2)) = 1.$$
This means that through a general point of $\sec_5(V_4^2)$ there is a $1$-dimensional family of $4$-planes that are $5$-secant to $V_{4}^2$. This reflects the fact that the expected dimension of $\sec_5(V_4^2)$ is $\expdim(\sec_5(V_4^2)) = 14$  while $\dim(\sec_5(V_4^2)) = 13$, that is the $5$-secant defect of $V_4^2$ is $\delta_5(V_4^2) = \expdim(\sec_5(V_4^2))-\dim(\sec_5(V_4^2)) = 1$.\\  
Now, $\Hilb_{5}(\mathbb{P}^{2*})$ is smooth. The fibers of $\phi$ over $U$ are open Zariski subsets of $\mathbb{P}^4$. So $\mathcal{X}$ is smooth and irreducible. Therefore, for $F$ varying in an open Zariski subset of $\sec_5(V_4^2)$ the fiber $\psi^{-1}(F)$ is a smooth and irreducible curve. Now, for a general $F\in k[x_0,x_1,x_2]_{4}$, let us consider the variety 
\[
  \begin{tikzpicture}[xscale=1.5,yscale=-1.5]
    \node (A0_1) at (1, 0) {$\VSP_6(F,6):=\overline{\{(l,\{L_1,...,L_6\})\: | \: l\in\{L_1,...,L_6\}\in \VSP(F,6)^{o}\}}\subseteq \mathbb{P}^{2*}\times \VSP(F,6)$};
    \node (A1_0) at (0, 1) {$\mathbb{P}^{2*}$};
    \node (A1_2) at (2, 1) {$\VSP(F,6)$};
    \path (A0_1) edge [->]node [auto] {$\scriptstyle{g}$} (A1_2);
    \path (A0_1) edge [->]node [auto,swap] {$\scriptstyle{f}$} (A1_0);
  \end{tikzpicture}
  \]
Let $l\in\mathbb{P}^{2*}$ be a general linear form. Note that the fiber $f^{-1}(l)$ consists of the points $\{L_1,...,L_6\}\in \VSP(F,6)$ such that $l\in\{L_1,...,L_6\}$. Therefore, we can identify $f^{-1}(l)$ with the $\{L_1,...,L_5\}\in \Hilb_{5}(\mathbb{P}^{2*})$ such that $F-l^{4}$ can be decomposed as a linear combination of $L_1^4,...,L_5^4$. Note that, since $F\in\mathbb{P}^{14}$ is general, we have that also $F-l^{4}$ is general in $\sec_5(V_2^4)$, and
$$f^{-1}(l) \cong \psi^{-1}(F-l^{4}).$$
In particular $f^{-1}(l)$ is a smooth irreducible curve and, since $\dim(\VSP_6(F,6)) = 3$, we conclude that $f:\VSP_6(F,6)\rightarrow\mathbb{P}^{2*}$ is dominant.
Now, our aim is to study the fiber of $\psi$ over a general point $G\in \sec_5(V_4^2)$. We can write 
$$G = \sum_{i=1}^{5}\lambda_iL_i^{4},$$
and let $C\subset\mathbb{P}^{2*}$ be the conic through $L_1,...,L_5$. Its image $\Omega = \nu_4^2(C)\subset\mathbb{P}^{14}$ is a rational normal curve of degree eight. Let $\left\langle\Omega\right\rangle = H^{8}\cong\mathbb{P}^{8}$ be its linear span. Therefore, we have $G\in \left\langle L_1^4,...,L_5^4\right\rangle\subset H^{8}\subset\mathbb{P}^{14}$. Now, $G$ is general in $H^{8}$ and we can interpret it as the class of a general polynomial $T\in K[z_0,z_1]_{8}$. The $4$-planes passing through $G$ that are $5$-secant to $\Omega$ are parametrized by $\VSP(T,5)$. Since any such $4$-plane is in particular $5$-secant to $V_{4}^2$, we have $\VSP(T,5)\subseteq\psi^{-1}(G)$.\\
Now, by \cite[Theorem 3.1]{MMe} we have $\VSP(T,5)\cong\mathbb{P}^1$. Since $\psi^{-1}(G)$ is an irreducible curve we conclude that $\psi^{-1}(G)$ is indeed a rational curve.\\
Finally, since $f:\VSP_6(F,6)\rightarrow\mathbb{P}^{2*}$ is dominant and its general fiber $f^{-1}(l) \cong \psi^{-1}(F-l^{4})\cong\mathbb{P}^1$ is rational, by \cite[Corollary 1.3]{GHS} we have that $\VSP_6(F,6)$ is rationally connected. 
\end{proof}
\begin{Theorem}\label{A17}
The moduli space $\mathcal{A}_2(1,7)^{-}_{sym}$ of $(1,7)$-polarized abelian surfaces with a symmetric theta structure and an odd theta characteristic is birational to the variety $\VSP_6(F,6)$ where $F\in k[x_0,x_1,x_2]_4$ is a general quartic polynomial. In particular $\mathcal{A}_2(1,7)^{-}_{sym}$ is rationally connected, and hence its Kodaira dimension is $-\infty$
\end{Theorem}
\begin{proof}
By \cite[Proposition 5.4 and Corollary 5.6]{GP1} there exists a birational map
$$\alpha: \mathcal{A}_2(1,7)^{lev}\dasharrow \VSP(F,6)$$
for $F$ the Klein quartic curve. As already observed in \cite{GP1}, the Klein quartic is general in the sense of Mukai \cite{Mu1}, hence the variety $\VSP(F,6)$ is isomorphic to the VSP obtained for any other general quartic curve. The map $\alpha$ is constructed as follows. For a general $(1, 7)$-polarized abelian surface $A$ with a level structure, embedded in $\P H^0(A,L)\cong \P^6$ the set of its odd $2$-torsion points is exactly the intersection $A\cap \P_-^2$. It turns out that the dual lines $\{L_{1,A},...,L_{6,A}\}$ in $\P_-^2$ are elements of $\VSP(F,6)$, and this correspondence gives a birational map.  By construction, there exists a morphism $f^{-}:\mathcal{A}_2(1,7)^{-}_{sym}\rightarrow \mathcal{A}_2(1,7)^{lev}$ of degree $6$ forgetting the odd theta characteristic. Moreover, from Section \ref{seccinque}, we know that given $(A,\psi)\in \A_2(1,7)^{lev}$, the map $Th^-_{(1,7)}$ sends the 6 elements of ${f^-}^{-1}(A,\psi)$ to the six odd 2-torsion points in $\P_-^2$ using the identification $\Psi_-$ induced by the symmetric theta structure. Therefore there is a commutative diagram
   \[
  \begin{tikzpicture}[xscale=2.7,yscale=-1.7]
    \node (A0_0) at (0, 0) {$\mathcal{A}_2(1,7)^{-}_{sym}$};
    \node (A1_0) at (0, 1) {$\mathcal{A}_2(1,7)^{lev}$};
    \node (A1_1) at (1, 1) {$\VSP(F,6)$};
    \path (A0_0) edge [->,swap]node [auto] {$\scriptstyle{f^{-}}$} (A1_0);
    \path (A1_0) edge [->,dashed]node [auto] {$\scriptstyle{\alpha}$} (A1_1);
    \path (A0_0) edge [->,dashed]node [auto] {$\scriptstyle{\alpha^{-}}$} (A1_1);
  \end{tikzpicture}
  \]
where $\alpha^{-} = \alpha\circ f^{-}$ is a degree six dominant rational map sending a $(1, 7)$-polarized abelian surface $A$ with an odd theta characteristic to the set $\{L_{1,A},...,L_{6,A}\}$ determined by its odd $2$-torsion points.\\
Now, we have a degree six rational map $\psi:\VSP_6(F,6)\dasharrow \VSP(F,6)$ whose fiber over a general point $\{L_{1,A},...,L_{6,A}\}\in \VSP(F,6)$ consists of the six linear forms $L_{i,A}$ in the decomposition of $F$ given by $\{L_{1,A},...,L_{6,A}\}$ which in turn are identified with the six odd $2$-torsion points of the abelian surface $A$. Now, consider a general point $(A,\psi,L)$ of $\mathcal{A}_2(1,7)^{-}_{sym}$ over $(A,\psi)\in\A_2(1,7)^{lev}$. 
Then there exists a rational map
$$
\begin{array}{ccc}
\beta: \mathcal{A}_2(1,7)^{-}_{sym} & \dasharrow & \VSP_6(F,6)\\
\end{array}
$$
sending $(A,\psi,L)$ to the linear form in $\psi^{-1}(\{L_{1,A},...,L_{6,A}\})$ that corresponds to $Th^-_{(1,7)}(A,\psi,L)\in \P_-^2$. Therefore, we have a commutative diagram
  \[
  \begin{tikzpicture}[xscale=2.7,yscale=-1.7]
    \node (A0_0) at (0, 0) {$\mathcal{A}_2(1,7)^{-}_{sym}$};
    \node (A0_1) at (1, 0) {$\VSP_6(F,6)$};
    \node (A1_0) at (0, 1) {$\mathcal{A}_2(1,7)^{lev}$};
    \node (A1_1) at (1, 1) {$\VSP(F,6)$};
    \path (A0_0) edge [->,dashed]node [auto] {$\scriptstyle{\beta}$} (A0_1);
    \path (A0_0) edge [->,swap]node [auto] {$\scriptstyle{f^{-}}$} (A1_0);
    \path (A0_1) edge [->,dashed]node [auto] {$\scriptstyle{\psi}$} (A1_1);
    \path (A1_0) edge [->,dashed]node [auto] {$\scriptstyle{\alpha}$} (A1_1);
    \path (A0_0) edge [->,dashed]node [auto] {$\scriptstyle{\alpha^{-}}$} (A1_1);
  \end{tikzpicture}
  \]
hence the map $\beta:\mathcal{A}_2(1,7)^{-}_{sym}\dasharrow \VSP_6(F,6)$ is birational. Finally, by Theorem \ref{vsp6rc} we have that $\mathcal{A}_2(1,7)^{-}_{sym}$ is rationally connected.
\end{proof}
\subsubsection{The case $n=9$}
Let $L$ be a symmetric line bundle on $A$ representing a polarization of type $(1,9)$. The linear system $|L|^*$ embeds $A$ in $\P^8$. This embedding is invariant under the Schr\"{o}dinger action of the Heisenberg group, and under the involution $\imath$. More precisely, the space of quadrics on $\P^8$ is 45 dimensional and it decomposes into five isomorphic irreducible representations of $\mathcal{H}_{1,11}$. In particular, the ideal of quadrics $H^0(\P^8,\mathcal{I}_A(2))$ is a representation of weight $2$ (the center $\mathbb{C}^*$ acts via its character $t^2$) of the Heisenberg group. More precisely, $A$ is embedded as a projectively normal surface of degree $18$ which is in fact contained in $9$ quadrics. However, these $9$ quadrics do not generate the homogeneous ideal of $A$. The 5 irreducible representations are highlighted in the $5\times 9$ matrix $R_4$ 
$$
R_4 = \left(
\begin{array}{ccccccccc}
x_0^2 & x_1^2 & x_2^2 & x_3^2 & x_4^2 & x_5^2 & x_6^2 & x_7^2 & x_8^2\\ 
x_1x_8 & x_0x_2 & x_1x_3 & x_2x_4 & x_3x_5 & x_4x_6 & x_5x_7 & x_6x_8 & x_0x_7\\ 
x_2x_7 & x_3x_8 & x_0x_4 & x_1x_5 & x_2x_6 & x_3x_7 & x_4x_8 & x_0x_5 & x_1x_6\\ 
x_3x_6 & x_4x_7 & x_5x_8 & x_0x_6 & x_1x_7 & x_2x_8 & x_0x_3 & x_1x_4 & x_2x_5\\ 
x_4x_5 & x_5x_6 & x_6x_7 & x_7x_8 & x_0x_8 & x_0x_1 & x_1x_2 & x_2x_3 & x_3x_4
\end{array}
\right)
$$
We refrain from giving the details on the representation theoretical aspects of this object, which are developed thoroughly in \cite[Section 3]{GP3}. We just need to know two facts.
\begin{Proposition}\label{duefatti}
Each $9$-dimensional Heisenberg representation in the space of quadrics is spanned by the quadrics obtained as $v\cdot R_4$ ($v$ is a row vector) for some $v\in \P^4_+$.\\  
Furthermore, If $p\in \P^8$ and $v\in \P^4_+$ then $v\cdot R_4(p)=0$ if and only if $p$ is contained in the scheme cut out by the quadrics in the representation determined by $v$.
\end{Proposition}
The anti-invariant eigenspace $\P^3_-$ is defined by the equations $\{x_0= x_i+x_{9-i}= 0,\: \forall\: i=1,\dots 8\}$, hence we can take $x_1,\dots,x_4$ as coordinates. A direct computation shows that, when we restrict $R_4$ to $\P^3_-$, we get the following anti-symmetric matrix
$$
R_{4|\P^3_{-}} = \left(
\begin{array}{ccccc}
0 & x_1^2 & x_2^2 & x_3^2 & x_4^2\\ 
-x_1^2 & 0 & x_1x_3 & x_2x_4 & -x_3x_4\\ 
-x_2^2 & -x_1x_3 & 0 & -x_1x_4 & -x_2x_3\\ 
-x_3^2 & -x_2x_4 & x_1x_4 & 0 & -x_1x_2\\ 
-x_4^2 & x_3x_4 & x_2x_3 & x_1x_2 & 0 
\end{array}
\right)
$$
\begin{Theorem}\label{A19}
The moduli space $\mathcal{A}_2(1,9)^{-}_{sym}$ of $(1,9)$-polarized abelian surfaces with canonical level structure and an odd theta characteristic is rational.
\end{Theorem}
\begin{proof}
Let us consider the theta-null morphism
$$
\begin{array}{ccc}
Th_{(1,9)}^{-}:\mathcal{A}_2(1,9)^{-}_{sym} & \longrightarrow & \mathbb{P}_{-}^3\\
(A,L,\psi) & \longmapsto & \Psi^{-}(\Theta_{1,9}(0))
\end{array}
$$
It is clear that $\det(R_{4|\P^3_{-}})$ is identically zero. By Lemma \ref{d1ed2} and what we have observed in Section \ref{nullodd}, we see that the closure of $D_2$ is the full $\P_-^3$ space and $Th_{(1,9)}^{-}$ is dominant, that is the general point of $\P^3$ is an odd $2-$torsion point of a $(1,9)$-abelian surface with level structure embedded in $\P^8$. Following \cite[Section 3]{GP3} we consider the \textit{Steinerian} map (this is the classical name for a map mapping a linear system of matrices to their kernels)
$$
\begin{array}{ccc}
Stein_{1,9}:\mathbb{P}_{-}^3 & \dashrightarrow & \mathbb{P}_{+}^4\\
p & \longmapsto & \Ker(R_{4|\P^3_{-}}(p)).
\end{array}
$$ 
Let us recall from \cite[Section 6]{GP2} that for $v\in \P^4_+$, $v\cdot R_4=0$ if and only if $v\cdot R_{4|\P^3_-}=0$.
Hence, by Proposition \ref{duefatti} we see that the image of $p\in \mathbb{P}_{-}^3$ is the $v\in \P^4_+$ that determines the unique $H_9$-sub-representation of $H^0(\mathbb{P}^8, \mathcal{O}_{\mathbb{P}^8}(2))$ of quadrics containing $p$. The map $Stein_{1,9}$ is given by the $4\times 4$ pfaffians of the matrix $R_{4|\P^3_{-}}$. In coordinates we have $Stein_{(1,9)}(x_1,...,x_4) = (y_0,...,y_4)$ where
$$
\begin{array}{l}
y_0 = -x_1^2x_2x_3+x_2^2x_3x_4+x_1x_3x_4^2,\\ 
y_1 = x_1x_3^2-x_2x_3^3+x_1x_4^3,\\ 
y_2 = -x_1^3x_2+x_3^3x_4+x_2x_4^3,\\ 
y_3 = x_1^2x_2x_3-x_2^2x_3x_4-x_1x_3x_4^2,\\ 
y_4 = x_1x_3^3-x_1^3x_4-x_3^2x_4.
\end{array}  
$$
Therefore the image of $Stein_{1,9}$ is contained in the hyperplane $\Pi = \{y_0+y_3 = 0\}\cong\mathbb{P}^3$ and the rational map $Stein_{1,9}:\mathbb{P}_{-}^3\dasharrow\Pi$ is dominant of degree $6$. Now, by \cite[Theorem 3.3]{GP3} the map $Stein_{1,9}$ induces an  isomorphism $\mathcal{A}_2(1,9)^{lev}\cong\Pi$, defined by mapping an abelian surface $A\subset\mathbb{P}^8$ to the point corresponding to the unique $H_9$-sub-representation of $H^0(\mathbb{P}^8, \mathcal{O}_{\mathbb{P}^8}(2))$ of quadrics containing $A$. Let $p\in\Pi$ be a general point, and $(A,\psi)$ the corresponding abelian surface with level structure. By Section \ref{nullodd}, the six points of the fiber $Stein_{1,9}^{-1}(p)$ correspond to the images via the theta-null map $Th_{1,9}^-:\A_{2}(1,9)^-_{sym} \to \P^3_-$ of the six possible choices of an odd theta characteristic for $(A,\psi)$. Hence we have a commutative diagram 
   \[
  \begin{tikzpicture}[xscale=2.9,yscale=-1.7]
    \node (A0_0) at (0, 0) {$\mathcal{A}_2(1,9)^{-}_{sym}$};
    \node (A0_1) at (1, 0) {$\mathbb{P}^3_{-}$};
    \node (A1_0) at (0, 1) {$\mathcal{A}_2(1,9)^{lev}$};
    \node (A1_1) at (1, 1) {$\Pi\cong\mathbb{P}^3$};
    \path (A0_0) edge [->]node [auto] {$\scriptstyle{Th_9^{-}}$} (A0_1);
    \path (A0_0) edge [->,swap]node [auto] {$\scriptstyle{f^{-}}$} (A1_0);
    \path (A0_1) edge [->,dashed]node [auto] {$\scriptstyle{Stein_{1,9}}$} (A1_1);
    \path (A1_0) edge [->,dashed]node [auto] {$\scriptstyle{\sim}$} (A1_1);
  \end{tikzpicture}
  \]
where $f^-$ is the $6$ to $1$ forgetful map. Therefore $Th_{(1,9)}^{-}$ is generically injective, and thus a birational map.
\end{proof}
\subsubsection{The case $n=11$}
Let $A$ be a general abelian surface with a symmetric line bundle $L$ representing a polarization of type $(1,11)$ and with canonical level structure $\psi$ (by Lemma \ref{sequence-split}, equivalently, a symmetric theta structure $\Psi$). The linear system $|L|^*$ embeds $A$ in $\P^{10}$ as a projectively normal surface of degree $22$ and sectional genus $12$. This embedding is invariant under the action of the Schr\"{o}dinger representation of the Heisenberg group. The ideal of quadrics $H^0(\P^{10},\mathcal{I}_A(2))$ is also a representation of weight $2$ of the Heisenberg group. This in turn implies that $H^0(\P^{10},\mathcal{I}_A(2))$ decomposes into irreducible components of dimension $11$. More precisely $H^0(\P^{10},\OO_{\P^{10}}(2))$ has dimension $66$ and decomposes into $6$ irreducible $11$-dimensional representation, isomorphic to the Schr\"{o}dinger representation. As we did in the $d=9$ case, let us then consider the $6\times 11$ matrix
$$
R_5 = \left(
\begin{array}{ccccccccccc}
x_0^2 & x_1^2 & x_2^2 & x_3^2 & x_4^2 & x_5^2 & x_6^2 & x_7^2 & x_8^2 & x_9^2 & x_{10}^2\\ 
x_1x_{10} & x_0x_2 & x_1x_3 & x_2x_4 & x_3x_5 & x_4x_6 & x_5x_7 & x_6x_8 & x_7x_9 & x_8x_{10} & x_0x_9\\ 
x_2x_9 & x_3x_{10} & x_0x_4 & x_1x_5 & x_2x_6 & x_3x_7 & x_4x_8 & x_5x_9 & x_6x_{10} & x_0x_7 & x_1x_8\\ 
x_3x_8 & x_4x_9 & x_5x_{10} & x_0x_6 & x_1x_7 & x_2x_8 & x_3x_9 & x_4x_{10} & x_0x_5 & x_1x_6 & x_2x_7\\ 
x_4x_7 & x_5x_8 & x_6x_9 & x_7x_{10} & x_0x_8 & x_1x_9 & x_2x_{10} & x_0x_3 & x_1x_4 & x_2x_5 & x_3x_6\\
x_5x_6 & x_6x_7 & x_7x_8 & x_8x_9 & x_9x_{10} & x_0x_{10} & x_0x_1 & x_1x_2 & x_2x_3 & x_3x_4 & x_4x_5
\end{array}
\right)
$$
Analogously to Proposition \ref{duefatti}, we have the following.
\begin{Proposition}\label{ventiduefatti}
Any $H_{11}$ irreducible sub-representation of $H^0(\P^{10},\OO_{\P^{10}}(2))$ is obtained by taking a linear combination of the rows with a vector of coefficients $v\in \P^5_+$, and taking the span of the resulting 11 quadratic polynomials.\\ 
Moreover, if $p\in\P^{10}$ and $v\in \P^5_+$, then $v\cdot R_5(p) = 0$ if and only if $p$ is contained in the scheme cut out by the $H_{11}$-sub-representation of quadrics determined by $v$. 
\end{Proposition}
The anti-invariant subspace $\P^4_-$ is defined as usual by $\{x_0 = x_i+x_{11-i} =  0,\ \forall i=1,\dots,10\}$ and the restriction of $R_5$ to $\P^4_-$ is the alternating matrix 
$$
R_{5|\P^4_-} = \left(
\begin{array}{cccccc}
0 & x_1^2 & x_2^2 & x_3^2 & x_4^2 & x_5^2\\ 
-x_1^{2} & 0 & x_1x_3 & x_2x_4 & x_3x_5 & -x_4x_5\\
-x_2^{2} & -x_1x_3 & 0 & x_1x_5 & -x_2x_5 & -x_3x_4\\
-x_3^{2} & -x_2x_4 & -x_1x_5 & 0 & -x_1x_4 & -x_2x_3\\
-x_4^{2} & -x_3x_5 & x_2x_5 & x_1x_4 & 0 & -x_1x_2\\
-x_5^{2} & x_4x_5 & x_3x_4 & x_2x_3 & x_1x_2 & 0
\end{array}
\right)
$$
\begin{Proposition}\label{A111}
The moduli space $\mathcal{A}_2(1,11)^{-}_{sym}$ of $(1,11)$-polarized abelian surfaces with canonical level structure and an odd theta characteristic is birational to the sextic hypersurface $X\subset\mathbb{P}^4$ given by $\det(R_{5|\P^4_-}) = 0$.
\end{Proposition}
\begin{proof}
As in the $d=9$ case, there exists a rational map
$$Stein_{11}: X\dasharrow \mathbb{G}(2,6)$$
mapping a point $p\in \P^3_-$ to the pencil of $H_{11}$-sub-representations of quadrics containing $p$, that is to the kernel of the matrix $R_{5|\P^4_-}$ evaluated in $p$. Recall from Remark \ref{level-image} that, for an abelian surface with a level structure $(A,\psi)$, the intersection set with each of the two eigenspaces is well defined. By Theorem \ref{dim-e-bl}, a general such surface intersects $\P^4_-$ along the $6$ odd $2$-torsion points. By \cite[Lemma 6.4]{GP2}, the six odd $2$-torsion points are mapped to the same point of $\mathbb{G}(2,6)$ via $Stein_{11}$ (actually they are the full fiber). Now, by Section \ref{nullodd} we know that these six points are the images, via the theta-null map
$$Th_{(1,11)}^{-}:\mathcal{A}_2(1,11)^{-}_{sym}  \longrightarrow  \mathbb{P}_{-}^4,$$
of the six choices $(A,\Psi,L)$ of an odd theta characteristic on $A$.
This means that the hypersurface $X = \{\det(R_{5|\P^4_-}) = 0\}$, that coincides with $D_2$, is the image of $Th_{(1,11)}^-$. By \cite[Theorem 2.2]{GP3}, $\mathcal{A}_2(1,11)^{lev}_{sym}$ is birational to the image $\Ima(Stein_{(1,11)})\subset \mathbb{G}(2,6)$.\\
Recalling now that $f^{-}:\mathcal{A}_2(1,11)^{-}_{sym}\rightarrow \mathcal{A}_2(1,11)^{lev}$ is the forgetful map of the odd theta characteristic, we have now the following commutative diagram
  \[
  \begin{tikzpicture}[xscale=3.3,yscale=-1.7]
    \node (A0_0) at (0, 0) {$\mathcal{A}_2(1,11)^{-}_{sym}$};
    \node (A0_1) at (1, 0) {$X\subset\mathbb{P}^{4}_{-}$};
    \node (A1_0) at (0, 1) {$\mathcal{A}_2(1,11)^{lev}$};
    \node (A1_1) at (1, 1) {$\Ima\subset\mathbb{G}(2,6)$};
    \path (A0_0) edge [->]node [auto] {$\scriptstyle{Th_{11}^{-}}$} (A0_1);
    \path (A1_0) edge [->,dashed]node [auto] {$\scriptstyle{\sim}$} (A1_1);
    \path (A0_1) edge [->,dashed]node [auto] {$\scriptstyle{Stein_{(1,11)}}$} (A1_1);
    \path (A0_0) edge [->,swap]node [auto] {$\scriptstyle{f^{-}}$} (A1_0);
  \end{tikzpicture}
  \]
Therefore, $Th_{11}^{-}$ is birational.
\end{proof}
\subsubsection{The case n=13}\label{A113}
By \cite[Theorem 6.5]{GP2}, the map
$$\Theta_{13}:\mathcal{A}_2(1,13)^{lev}\dasharrow\mathbb{G}(3,7)$$
mapping an abelian surface $A$ to the sub-representation of $H^{0}(A,\mathcal{O}_A(2))$ given by $H^{0}(A,\mathcal{I}_{A}(2))$, is birational onto its image. As usual, we have the following commutative diagram
 \[
  \begin{tikzpicture}[xscale=3.3,yscale=-1.7]
    \node (A0_0) at (0, 0) {$\mathcal{A}_2(1,13)^{-}_{sym}$};
    \node (A0_1) at (1, 0) {$\mathbb{P}^{5}_{-}$};
    \node (A1_0) at (0, 1) {$\mathcal{A}_2(1,13)^{lev}$};
    \node (A1_1) at (1, 1) {$\Ima\subset\mathbb{G}(3,7)$};
    \path (A0_0) edge [->]node [auto] {$\scriptstyle{Th_{(1,13)}^{-}}$} (A0_1);
    \path (A1_0) edge [->,dashed]node [auto] {$\scriptstyle{\Theta_{13}}$} (A1_1);
    \path (A0_1) edge [->,dashed]node [auto] {$\scriptstyle{6:1}$} (A1_1);
    \path (A0_0) edge [->,swap]node [auto] {$\scriptstyle{f^{-}}$} (A1_0);
  \end{tikzpicture}
  \]
and $Th_{13}^{-}$ is birational onto its image in $\mathbb{P}^{5}_{-}$. In this case $R_{6|\mathbb{P}^{5}_{-}}$ is a $7\times 7$ anti-symmetric matrix. In this case, $D_2\subset\mathbb{P}^{5}_{-}$ is the variety defined by the vanishing of the $6\times 6$ pfaffians of $R_{6|\mathbb{P}^{5}_{-}}$. Clearly, $\Ima(Th_{13}^{-})\subseteq D_2$. Furthermore, a computation in Macaulay2 \cite{Mc2} shows that $D_2$ is an irreducible $3$-fold of degree $21$, scheme-theoretically defined by the following three pfaffians
$$
\begin{array}{ll}
f_1 =  & -x_1^2x_3^3x_4+x_1x_2^3x_4^2-x_1^4x_4x_5+x_1x_2x_3x_4^2x_5-x_2^3x_3x_5^2+x_1x_3x_5^4-x_2x_3x_4^3x_6+x_1^2x_2^2x_5x_6+\\  
& x_3^4x_5x_6-x_1x_4^3x_5x_6-x_1^2x_3x_5^2x_6-x_2^2x_4x_5^2x_6+x_1^3x_3x_6^2+x_1x_2^2x_4x_6^2+x_2x_3^2x_4x_6^2,\\
f_2 =  & -x_1x_2x_3^4+x_2^4x_3x_4+x_1x_3^2x_4^3-x_1^3x_2x_3x_5-x_2^2x_3^2x_4x_5-x_2x_4^2x_5+x_3^3x_4x_5^2+x_1^2x_2x_3x_4x_6+\\  
& x_1^3x_4x_5x_6+x_1x_4^2x_5^2x_6-x_1x_2^2x_5x_6^2-x_2x_3^2x_5x_6^2+x_1x_3x_5^2x_6^2-x_1^2x_3x_6^3-x_2^2x_4x_6^3,\\ 
f_3 =  & -x_1^2x_2x_3^3+x_1x_2^4x_4-x_1^4x_2x_5+x_1^2x_2x_4x_5^2+x_1x_3^2x_4x_5^2-x_2x_4^2x_5^3-x_1^2x_3^2x_4x_6-x_2^2x_3x_4^2x_6+\\  
& x_3^2x_4^2x_5x_6-x_2^3x_5^2x_6+x_1x_5^4x_6-x_1x_2x_3x_5x_6^2+x_3^3x_5x_6^2+x_2x_3x_4x_6^3+x_1x_4x_5x_6^3.
\end{array} 
$$
Hence $\Ima(Th_{13}^{-})= D_2$ and $\mathcal{A}_2(1,13)^{-}_{sym}$ is birational to $D_2$.

\subsection{Polarizations of type $(1,n)$ with $n$ even.}
Let $A$ be a $(1,2d)$-polarized abelian surface with a level structure. As it was pointed out in Remark \ref{level-image}, the two intersection sets of $A$ with the eigenspaces are well defined. Let $\mathcal{H}_{1,2d}$ be the finite Heisenberg group defined in Definition \ref{finiteheis}, $\sigma$ and $\tau$ the two generators such that $\sigma(x_i) = x_{i-1}$, $\tau(x_i) = \xi^{-1}x_{i}$ with $\xi = e^{\pi i/d}$, on the homogeneous coordinates $x_0,\dots,x_{2d-1}$ on $\P(H^0(A,L))^{*}$.\\
Both $\sigma^d$ and $\tau^d$ act on the ${-1}$-eigenspace $\P_-$, and this defines a $\Z/2\Z \times \Z/2\Z$ action on $\P_-$. If $A\subset \P^{2d-1}$ is a Heisenberg invariant abelian surface, by Propositions \ref{actions} and \ref{dim-e-bl} we have $A[2]^-=A\cap\P_-$ and this set is a $\Z/2\Z \times \Z/2\Z$ orbit on $\P_-$. Let us now define the $d\times d$ matrix
$$
(M_d)_{i,j}:= x_{i+j}y_{i-j}+x_{i+j+d}y_{i-j+d},\ \ \ 0\leq i,j\leq d-1,
$$
where the indices are modulo $2d$. We will need to keep in mind the following \cite[Theorem 6.2]{GP2}.
\begin{Theorem}\label{homoidea}
Let $A\subset \P^{2d-1}$ a general Heisenberg invariant, $(1,2d)$-polarized abelian surface, and $y\in A \cap \P_-$. Then, the $4\times 4$ pfaffians of the anti-symmetric minors of the matrices
\begin{equation}\label{matrame}
\left\{\begin{array}{ll}
M_5(x,y);M_5(x,\sigma^5(y));M_5(x,\tau^5(y); & \mathrm{\textit{if}\ \textit{d}=5;}\\
M_d(x,y);M_d(x,\sigma^d(y)); & \mathrm{\textit{if}\ \textit{d}\geq 7,\ \textit{d}\ \textit{odd};}\\
M_6(x,y);M_6(\sigma(x),y);M_6(\tau(x),y); & \mathrm{\textit{if}\ \textit{d}=6;}\\
M_d(x,y);M_d(\sigma(x),y); & \mathrm{\textit{if}\ \textit{d}\geq 8,\ \textit{d}\ \textit{even};}
\end{array}\right.
\end{equation}
generate the homogeneous ideal of $A$.
\end{Theorem}
\subsubsection{The case $n=8$}
Let $A$ be a $(1,8)$-polarized abelian surface. We are now in what so far we have called the intermediate case. The line bundle $L$ corresponding to the polarization induces an embedding $A\rightarrow\mathbb{P}^{7}\cong\mathbb{P}(H^{0}(A,L)^{*})$ of degree $16$. Let us fix homogeneous coordinates $x_0,...,x_7$ on $\mathbb{P}^7$, and consider the usual action of the Heisenberg group $\mathcal{H}_{1,8}$, where the two generators operate as $\sigma(x_i) = x_{i-1}$, $\tau(x_i) = \xi^{-1}x_{i}$ with $\xi = e^{\pi i/4}$.\\
The standard involution $(x_i) \mapsto (x_{-i})$ on $\mathbb{Z}_{8}$ induces on $A$ the involution $\imath$. The eigenspaces $\mathbb{P}^2_{-}$ and $\mathbb{P}^4_{+}$ are, respectively, defined in $\P (H^0(A,L)^*)$ by $\{x_0 = x_4 = x_1+x_7 = x_2+x_6 = x_3+x_5 = 0\}$ and $\{x_1-x_7 = x_2-x_6 = x_3-x_5 = 0\}$.\\
Let us now consider the subgroup $\mathcal{H}^{'} := \left\langle \sigma^{4},\tau^4\right\rangle \cong (\Z/2\Z)^2\subset \mathcal{H}_{(1,8)}$. As we have observed, $\mathcal{H}^{'}$ acts on $\mathbb{P}^2_{-}$ and if $A$ is an abelian surface embedded in $\P (H^0(A,L)^*)$, then the four $2$-torsion points of $A\cap \P_-^2$ consist of an $\HH'$-orbit on $\P_-^2$. Furthermore, as it is remarked in \cite[Section 6]{GP1}, if $y_1,y_2,y_3$ are homogeneous coordinates on $\mathbb{P}^2_{-}$ we can embed $\mathbb{P}_{-}^2/\mathcal{H}^{'}$ in $\mathbb{P}^{3}$ by the map
\begin{equation}\label{mappone}
\begin{array}{ccc}
\mathbb{P}^2_{-}/\mathcal{H}^{'} & \longrightarrow & \mathbb{P}^3\\
\left[y_1:y_2:y_3\right] & \longmapsto & [2y_1y_3:-y_2^2:y_1^2+y_3^2:-y_2^2]
\end{array}
\end{equation}
Therefore, the image of $\mathbb{P}^2_{-}/\HH^{'}$ in $\mathbb{P}^{3}$ is the plane $\{w_1-w_3 = 0\}$, where $w_0,w_1,w_2,w_3$ are the homogeneous coordinates of $\mathbb{P}^3$. The quotient morphism $\mathbb{P}^2_{-}\rightarrow \mathbb{P}^2_{-}/\HH^{'}\cong\mathbb{P}^2$ is finite of degree four. We keep denoting by $y_1,y_2,y_3$ and $w_0,w_1,w_2$ the homogeneous coordinates on $\mathbb{P}^2_{-}$ and $\mathbb{P}^2_{-}/\HH^{'}$ respectively.\\
Let us now recall briefly a few results from \cite[Section 6]{GP1}. Let $\mathcal{A}_2(1,8)^{lev}$ be as usual the moduli space of $(1,8)$-polarized abelian surfaces with canonical level structure. There exists a dominant map 
$$\Theta_{8}:\mathcal{A}_2(1,8)^{lev}\rightarrow\mathbb{P}^2_{-}/\HH^{'}\cong\mathbb{P}^2$$    
associating to a $(1,8)$-polarized abelian surface with canonical level structure the set of its odd $2$-torsion points. For a general point $y\in\mathbb{P}^2_{-}$, let $V_{8,y}\subset\mathbb{P}^7$ denote the subscheme defined by the quadrics of $\mathbb{P}^7$ invariant under the action of $\HH^{'}$ and vanishing on the Heisenberg orbit of $y$. For a general $y\in\mathbb{P}^2_{-}$, $V_{8,y}$ is a Calabi-Yau complete intersection of type $(2,2,2,2)$ with exactly $64$ nodes. The fibre of $\Theta_8$ over a general point $y\in\mathbb{P}^2_{-}/\mathcal{H}^{'}$ corresponds to a pencil of abelian surfaces contained in the singular Calabi-Yau complete intersection $V_{8,y}$. Furthermore, by \cite[Theorem 6.8]{GP1} $\mathcal{A}_2(1,8)^{lev}$ is birational to a conic bundle over $\mathbb{P}^2_{-}/\mathcal{H}^{'}\cong\mathbb{P}^2$ with discriminant locus contained in the curve
$$\Delta=\{2w_1^4-w_0^3w_2-w_0w_2^3\}.$$ 
\begin{Proposition}\label{disc}
The discriminant of the conic bundle 
$$\Theta_{8}:\mathcal{A}_2(1,8)^{lev}\rightarrow\mathbb{P}^2_{-}/\mathcal{H}^{'}\cong\mathbb{P}^2$$ 
is the whole curve $\Delta$. 
\end{Proposition}
\begin{proof}
Recall that the fibre of $\Theta_8$ over a point $y\in\mathbb{P}^2_{-}/\mathcal{H}^{'}$ corresponds to a pencil of abelian surfaces contained in the singular Calabi-Yau complete intersection $V_{8,y}$. By \cite[Section 1.2]{Pa}, the equation defining the complete intersection $V_{8,y}$ in $\mathbb{P}^7$ are the following:
$$
\begin{array}{ccc}
f & = & y_1y_3(x_0^2+x_4^2)-y_2^2(x_1x_7+x_3x_5)+(y_1^2+y_3^2)x_2x_6,\\ 
\sigma(f) & = & y_1y_3(x_1^2+x_5^2)-y_2^2(x_2x_0+x_4x_6)+(y_1^2+y_3^2)x_3x_7,\\ 
\sigma^2(f)& = & y_1y_3(x_2^2+x_6^2)-y_2^2(x_3x_1+x_5x_7)+(y_1^2+y_3^2)x_4x_0,\\ 
\sigma^3(f) & = & y_1y_3(x_3^2+x_7^2 )-y_2^2(x_4x_2+x_6x_0)+(y_1^2+y_3^2)x_5x_1.
\end{array} 
$$
Consider the point $[y_1:y_2:y_3] = [0:0:1]$, which is mapped to the point $[0:0:1]\in \mathbb{P}^2_{-}/\mathcal{H}^{'}$. We see that for $y = [0:0:1]$ the variety $V_{8,y}$ is given by $\{x_2x_6 = x_3x_7 = x_0x_4 = x_1x_5 = 0\}$. Hence $V_{8,y}$ is the union of $16$ linear subspaces of dimension three in $\mathbb{P}^7$. In particular, $V_{8,y}$ does not contain a pencil of abelian surfaces and the conic bundle structure of $\Theta_{8}:\mathcal{A}_2(1,8)^{lev}\rightarrow\mathbb{P}^2_{-}/\mathcal{H}^{'}\cong\mathbb{P}^2$ degenerates on $[0:0:1]\in\Delta$. Therefore the discriminant locus of this conic bundle is non-empty, hence it is a curve. By \cite[Theorem 6.8]{GP1} we know that the discriminant locus is contained in the curve $\Delta$. Now, it is enough to observe that $\Delta$ is smooth, in particular irreducible, to conclude that the discriminant locus is exactly $\Delta$. 
\end{proof}
In particular, since $\deg(\Delta) = 4$, as it is remarked in \cite[Theorem 6.8]{GP1}, the moduli space $\mathcal{A}_2(1,8)^{lev}$ is rational thanks to the classification of conic bundles from \cite{Be}.
\begin{Theorem}\label{18}
The moduli space $\mathcal{A}_2(1,8)_{sym}^{-}$ of $(1,8)$-polarized abelian surfaces with a symmetric theta structure and an odd theta characteristic is birational to a conic bundle over $\mathbb{P}^2$ whose discriminant locus is a smooth curve of degree eight. In particular $\mathcal{A}_2(1,8)_{sym}^{-}$ is unirational but not rational.
\end{Theorem}
\begin{proof}
In Section \ref{nullint} we defined the morphism
$$
\begin{array}{ccc}
Th_{(1,8)}^{-}:\mathcal{A}_2(1,8)_{sym}^{-} & \longrightarrow & \mathbb{P}_{-}^2.\\
\end{array}
$$
We claim that it fits in the following commutative diagram
  \[
  \begin{tikzpicture}[xscale=2.9,yscale=-1.7]
    \node (A0_0) at (0, 0) {$\mathcal{A}_2(1,8)_{sym}^{-}$};
    \node (A0_1) at (1, 0) {$\mathbb{P}^2_{-}\cong\mathbb{P}^2$};
    \node (A1_0) at (0, 1) {$\mathcal{A}_2(1,8)^{lev}$};
    \node (A1_1) at (1, 1) {$\mathbb{P}^2_{-}/H^{'}\cong\mathbb{P}^2$};
    \path (A0_0) edge [->]node [auto] {$\scriptstyle{Th_{8}^{-}}$} (A0_1);
    \path (A0_0) edge [->,swap]node [auto] {$\scriptstyle{f^{-}}$} (A1_0);
    \path (A0_1) edge [->]node [auto] {$\scriptstyle{}$} (A1_1);
    \path (A1_0) edge [->]node [auto] {$\scriptstyle{\Theta_8}$} (A1_1);
  \end{tikzpicture}
  \]
This is due to the fact (see Section \ref{nullint}, Proposition \ref{actions} and Proposition \ref{dim-e-bl}) that, given an abelian surface with level structure $(A,\psi)$, the $4$ choices of symmetric theta structure that induce $\psi$, plus the odd theta characteristic (which is unique) are mapped exactly to the $4$ points of intersection of $A$ with  $\P^2_-$. Therefore, the finite morphism $f^{-}$ maps fibers of $Th_8^{-}$ to fibers of $\Theta_8$, and $Th_8^{-}:\mathcal{A}_2(1,8)_{sym}^{-}\rightarrow\mathbb{P}^2_{-}\cong\mathbb{P}^2$ is a conic bundle. By Proposition \ref{disc} the discriminant of this conic bundle is the inverse image of the curve $\Delta=\{2w_1^4-w_0^3w_2-w_0w_2^3\}$ via the projection $\mathbb{P}^2\rightarrow\mathbb{P}^2_{-}/H^{'}$. By substituting the equations (\ref{mappone}), we get that the discriminant is the curve
$$\Delta^{'} = \{2y_2^8-14y_1^5y_3^3-14y_1^3y_3^5-2y_1^7y_3-2y_1y_3^7 = 0\}.$$
Note that $\Delta^{'}$ is a smooth plane curve of degree eight. Since $\deg(\Delta^{'})\geq 6$, by \cite[Theorem 4.9]{Be}, the variety $\mathcal{A}_2(1,8)_{sym}^{-}$ is not rational. On the other hand, $\deg(\Delta^{'})\leq 8$, and by \cite[Corollary 1.2]{Me} $\mathcal{A}_2(1,8)_{sym}^{-}$ is unirational.
\end{proof}
\subsubsection{The case $n = 10$}\label{10}
An argument analogous to the one used in the proof of Theorem \ref{18} works in the case $n = 10$ as well. Here the negative eigenspace is of dimension three. Following \cite[Theorem 6.2]{GP2}, we have:
\begin{Theorem}\label{thm62}
Let $d$ be an even positive integer. The morphism
$$\Theta_{d}:\mathcal{A}_2(1,d)^{lev}\rightarrow\mathbb{P}^{\frac{d}{2}-2}_{-}/\mathbb{Z}_2\times\mathbb{Z}_2$$
mapping an abelian surface to the orbit of its odd $2$-torsion points is birational onto its image for $d\geq 10$.
\end{Theorem}
Thus in particular this is true for 
$$\Theta_{10}:\mathcal{A}_2(1,10)^{lev}\rightarrow\mathbb{P}^{3}_{-}/\mathbb{Z}_2\times\mathbb{Z}_2$$
The upshot is that $\mathcal{A}_2(1,10)^{lev}$ is rational. In fact, the restriction of the matrix $M_5$ from equation (\ref{matrame}) is a $5\times 5$ anti-symmetric matrix with linear entries on $\P_-^3$, hence its determinant is never maximal. Therefore, the sets of odd $2$-torsion points $A\cap \P^3_-$ cover the whole $\P_-^3$, when $A$ moves inside $\A_2(1,10)^{lev}$. We have the now familiar commutative diagram  
\[
  \begin{tikzpicture}[xscale=2.9,yscale=-1.7]
    \node (A0_0) at (0, 0) {$\mathcal{A}_2(1,10)_{sym}^{-}$};
    \node (A0_1) at (1, 0) {$\mathbb{P}_{-}^3$};
    \node (A1_0) at (0, 1) {$\mathcal{A}_2(1,10)^{lev}$};
    \node (A1_1) at (1, 1) {$\mathbb{P}_{-}^3/\mathbb{Z}_2\times\mathbb{Z}_2$};
    \path (A0_0) edge [->]node [auto] {$\scriptstyle{Th_{(1,10)}^{-}}$} (A0_1);
    \path (A0_0) edge [->,swap]node [auto] {$\scriptstyle{f^{-}}$} (A1_0);
    \path (A0_1) edge [->]node [auto] {$\scriptstyle{}$} (A1_1);
    \path (A1_0) edge [->]node [auto] {$\scriptstyle{\Theta_{10}}$} (A1_1);
  \end{tikzpicture}
  \]
with $4$ to $1$ vertical arrows. Hence, $Th_{10}^{-}$ is birational, and $\mathcal{A}_2(1,10)_{sym}^{-}$ is rational.
\subsubsection{The case $n=12$}
In this section we consider the moduli space $\mathcal{A}_2(1,12)_{sym}^{-}$ of $(1,12)$-polarized abelian surfaces with a symmetric theta structure and an odd theta characteristic. By \cite[Section 2]{GP4} if $A\subset\mathbb{P}^{11}$ is an $H_{12}$-invariant abelian surface of type
$(1, 12)$, and $y \in A$, then the matrix $M_{6}(x,y)$ from equation (\ref{matrame}) has rank at most two on $A$. In particular, the matrix $M_{6}(x,x)$ has rank at most two for any $x\in A\cap\mathbb{P}^4_{-}$. Now, $\mathbb{P}^{4}_{-}$ is defined by $\mathbb{P}^{4}_{-}=\{x_0 = x_6 =$ $ = x_5+x_7 = x_4+x_8 = x_3+x_9 = x_2+x_{10} = x_{1}+x_{11}=0\}\subset\mathbb{P}^{11}.$ Therefore the upper left $4\times 4$ block of $M_{6}(x,x)$ is 
$$
\left(
\begin{array}{cccc}
0 & -x_1^2-x_5^2 & -x_2^2-x_4^2 & -2x_3^2\\
x_1^2+x_5^2 & 0 & -x_1x_3-x_3x_5 & -2x_2x_4\\
x_2^2+x_4^2 & x_1x_3+x_3x_5 & 0 & -2x_1x_5\\
2x_3^2 & 2x_2x_4 & 2x_1x_5 & 0
\end{array}
\right)
$$
and its pfaffian is
$$P = 2(x_1x_3^3+x_3^3x_5-x_2^3x_4-x_2x_4^3+x_1^3x_5+x_1x_5^3).$$
We denote by $X_4^3$ the quartic $3$-fold
$$X_4^3 = \{x_1x_3^3+x_3^3x_5-x_2^3x_4-x_2x_4^3+x_1^3x_5+x_1x_5^3 = 0\}\subset\mathbb{P}^{4}_{-}.$$
By Theorem \ref{thm62}, there exists a birational map 
$$\Theta_{12}:\mathcal{A}_2(1,12)^{lev}\dasharrow X_4^3/\mathbb{Z}_{2}\times\mathbb{Z}_2$$
mapping $A$ to the $(\mathbb{Z}_{2}\times\mathbb{Z}_2)$-orbit of $A \cap \mathbb{P}^{4}_{-}$. In this case the action of $(\mathbb{Z}_{2}\times\mathbb{Z}_2)$ on $\mathbb{P}^4_{-}$ is given by 
$$
\begin{array}{lcc}
\sigma^{6}(x_1,x_2,x_3,x_4,x_5) = (x_5,x_4,x_3,x_2,x_1),\\
\tau^{6}(x_1,x_2,x_3,x_4,x_5) = (x_1,-x_2,x_3,-x_4,x_5).
\end{array} 
$$
By \cite[Theorem 2.2]{GP4} the quotient $X_4^3/\mathbb{Z}_{2}\times\mathbb{Z}_2$ is birational to the complete intersection $\mathbb{G}(1,3)\cap Q\subset\mathbb{P}^5$, where $\mathbb{G}(1,3) = \{y_0y_5-y_1y_4+y_2y_3=0\}$ is the Grassmannian of lines in $\mathbb{P}^3$, and $Q$ is the quadric given by $Q = \{y_0y_2-y_3^2-2y_2y_5=0\}$. Therefore $\mathcal{A}_2(1,12)^{lev}$ is rational. In the following subsection, we will show that the quartic $X_4^3$ is unirational, not rational, and birational to $\A_2(1,12)^-_{sym}$.
\subsubsection*{A unirational smooth quartic $3$-fold}
Let $X\subset\mathbb{P}^4$ be a smooth quartic hypersurface. By adjunction we have that $\omega_{X}\cong\mathcal{O}_{X}(-1)$, hence $X$ is Fano. The rational chain connectedness and, in characteristic zero, the rational connectedness of Fano varieties has been proven in \cite{Ca} and \cite{KMM}.\\
Clearly a unirational variety is rationally connected. However, establishing if the classes
of unirational and rationally connected varieties are actually distinct is a long-standing open problem in birational geometry.\\
We are interested in the quartic $3$-fold $X_4^3\subset\mathbb{P}^4_{-}$. We may write its equation as
$$X_{4}^3 = \{x_0x_2^3+x_2^3x_4-x_1^3x_3-x_1x_3^3+x_0^3x_4+x_0x_4^3 = 0\}$$
by shifting the indices of the homogeneous coordinates on $\mathbb{P}^4_{-}$.\\
By \cite{IM} for any smooth quartic $3$-fold $X\subset\mathbb{P}^3$ we have $\Bir(X) = \Aut(X)$. In particular, $X$ is not rational. Furthermore, this result was extended to nodal $\mathbb{Q}$-factorial quartic $3$-folds in \cite{CM} and \cite{Me1}. This gave new counterexamples to the famous L\"uroth problem in dimension three. On the other hand, Segre \cite{Se} gave a criterion for the unirationality of a smooth quartic $3$-fold and produced an example as well. In the rest of this section we will apply Segre's criterion to the quartic $X_4^3$ and prove the unirationality of $\A_2(1,12)_{sym}^-$. This criterion consists of the following steps:
\begin{itemize}
\item[-] first, we consider the open subset $X_0\subseteq X$ of points $x\in X$ such that there are at most finitely many triple tangent of $X$ through $x$,
\item[-] we consider the projectivized tangent bundle $\mathbb{P}(TX_{0})\rightarrow X_0$, and the subscheme $Y_0\subset\mathbb{P}(TX_{0})$ parametrizing triple tangents to $X_0$. Then we define a rational map
$$f:Y_0\dasharrow X$$
mapping a triple tangent to its fourth point of intersection with $X$.
\item[-] we construct a rational $3$-fold $Z_0\subset Y_0$ such that $f_{|Z_0}$ is finite.
\end{itemize}
\begin{Proposition}\label{3uni}
The quartic $3$-fold $X_{4}^3$ is unirational but not rational.
\end{Proposition}
\begin{proof}
We will denote $X_{4}^3$ simply by $X$. It is easy to check that $X$ is smooth. Therefore $X$ is not rational \cite{IM}. Our strategy, in order to prove the unirationality of $X$, consists in applying the unirationality criterion of \cite[Section 4]{Se}. A line $L\subset\mathbb{P}^4$ will be called a triple tangent to $X$ at a point $x\in X$ if either $x\in L\subset X$ or the intersection $L\cap X$ is of the form $3x + y$.\\
Let us consider the point $x = [10:2:1:1:0]$. We have that
$$\mathbb{T}_xX = \{x_0-13x_1+30x_2-14x_3+1001x_4 = 0\}.$$
Using \cite{Mc2} it is straightforward to check that the intersection $S(x) = X\cap\mathbb{T}_pX$ is an irreducible and reduced degree four surface, the point $x$ has multiplicity two on $S(x)$, and the quadratic tangent cone to $S(x)$ at $x$ is irreducible and reduced as well. Note that the triple tangents to $X$ at $x$ are the generators of the quadratic tangent cone to $S(x)$ at $x$. Now, assume that infinitely many triple tangents lie in $X$. Then the tangent cone lie in $S(x)$ which is irreducible and reduced. Therefore we get a contradiction and only finitely many triple tangents can lie in $X$.\\
It is well known that the subset $X_0\subseteq X$ of points with this property is a dense open subset of $X$. Now, let us consider the projectivized tangent bundle $\mathbb{P}(TX_{0})\rightarrow X_0$. Let $Y_0\subset\mathbb{P}(TX_{0})$ be the subscheme parametrizing triple tangents to $X_0$, and let $\pi:Y_0\rightarrow X_0$ be the projection. Note that if $x\in X_0$ the fiber $\pi^{-1}(x)$ is isomorphic to the base of the quadratic tangent cone to $S(x)$, that is $\pi^{-1}(x)\cong\mathbb{P}^1$. Now, only finitely many points on the fiber $\pi^{-1}(x)$ correspond to triple tangents contained in $X$. Therefore we can define a rational map $f:Y_0\dasharrow X$ mapping a triple tangent to its fourth point of intersection with $X$. Now, following \cite{Se} we would like to construct a rational $3$-fold $Z_0\subset Y_0$ such that $f_{|Z_0}$ is finite. Here comes the core part of the construction.\\
Let us consider the hyperplane $H_{4} = \{x_4=0\}$. Note that $H_{4} = \mathbb{T}_{q}X$ where $q = [1:0:0:0:0]$. The intersection $H_{4}\cap X$ is the surface 
$$S = \{G = x_0x_2^3-x_1^3x_3-x_1x_3^3 = 0\}\subset H_{4}\cong\mathbb{P}^3.$$
The partial derivatives of $G$ are
$$\frac{\partial G}{\partial x_0} = x_2^3,\: \frac{\partial G}{\partial x_1} = -3x_1^2x_3-x_3^3,\: \frac{\partial G}{\partial x_2} = 3x_0x_2^2,\: \frac{\partial G}{\partial x_3} = -x_1^3-3x_1x_3^2$$
and the Hessian matrix of $G$ is
$$
H(G) = \left(
\begin{array}{cccc}
0 & 0 & 3x_2^2 & 0\\
0 & -6x_1x_3 & 0 & -3x_1^2-3x_3^2\\
3x_2^2 & 0 & 6x_0x_2 & 0\\
0 & -3x_1^2-3x_3^2 & 0 & -6x_1x_3
\end{array}
\right)
$$
We see that $\dim(\Sing(S)) = 0$, so $S$ is irreducible. Furthermore, on the point $[1:0:0:0]$ all the first partial derivatives and the Hessian matrix vanish. On the other hand $\frac{\partial^{3}G}{\partial x_2^3}(1,0,0,0)\neq 0$, then $[1:0:0:0]$ is a singular point of multiplicity exactly three for $S$. In particular, since $\deg(S) = 4$ projecting from $[1:0:0:0]$ we see that $S$ is rational. Finally $x = [10:2:1:1:0]\in S$, and $S\cap X_0\neq\emptyset$.\\
Now, we define $Z_0 := \pi^{-1}(S)$. The general fiber of $\pi_{|Z_0}:Z_0\rightarrow S$ is a smooth rational curve. In order to prove that $Z_0$ is rational it is enough to show that $\pi_{|Z_0}$ admits a rational section. Let $x\in S$ be a smooth point. Then $\mathbb{T}_xS$ intersects the quadratic tangent cone to $S(x)$ in the two generators. In turn the two generators give two points on the fiber of $Z_0$ over $x$. We denote by $D\subset Z_0$ the closure of the locus of these pairs of points. Note that $D$ is a double section of $\pi_{|Z_0}:Z_0\rightarrow S$. Now, the surface of triple tangents of $S$ is given by the following two equations
$$\sum_{i=0}^3\frac{\partial G}{\partial x_i(x)}x_i = 0,\: \sum_{i,j=0}^3\frac{\partial^2 G(x)}{\partial x_ix_j}x_ix_j = 0$$
for $x$ varying in $S$. Therefore, the discriminant of the equation defining the two triple tangents at a general point $x\in S$ is the determinant of the Hessian $H(G)$ up to a quadratic multiple. We have
$$\det(H(G)) = (9x_2^2(x_1^2-x_3^2))^2.$$
Therefore, the surface of triple tangents of $S$ splits in two components $D = D_0\cup D_1$, and each component gives a rational section of $\pi_{|Z_0}:Z_0\rightarrow S$. We conclude that $Z_0$ is rational.
Now, we consider the restriction
$$f_{|D_0}:D_0\dasharrow S.$$
Note that $D_0$ is the surface given by 
$$\sum_{i=0}^3\frac{\partial G(x)}{\partial x_i}x_i = 0,\: \sum_{i=0}^3\alpha_i(x)x_i = 0.$$
for $x$ varying in $S$, where the $\alpha_i$ are determined by the splitting $D = D_0\cup D_1$. For instance, if $x = [10:2:1:1]$ the triple tangent $L$ corresponding to the point of $D_0$ over $x$ is given by  
$$L = \{2x_1-5x_2+x_3 = 2x_0-5x_2-15x_3 = 0\}$$
and $L$ intersects $S$ in $x$ with multiplicity three and in the fourth point $[65:1:2:8]$. Now, a standard computation shows that for a general point $y\in S$ there exist a point $x\in S$ and a triple tangent $L_x$ to $S$ at $x$ such that $y\in  L_x$. In other words the rational map $f_{|D_0}:D_0\dasharrow S$ is dominant.\\
Let us come back to the rational map $f_{|Z_0}:Z_0\dasharrow X$. Let us assume that $f_{|Z_0}$ is not dominant.  Since $D_0\subset Z_0$ and $D_0$ is dominant on $S$ we have that $\overline{f_{|Z_0}(Z_0)}$ is an irreducible surface containing $S$. Therefore, $\overline{f_{|Z_0}(Z_0)} = S$. Now, let $x\in S$ be any smooth point. Then $S(x)\neq S$, and the general generator of the quadratic tangent cone to $S(x)$ in $x$ does not lie on the hyperplane $H_4$ cutting $S$ on $X$. In particular, the fourth point of intersection of such a general generator with $X$ does not lie in $S$. This is a contradiction. We conclude that $f_{|Z_0}:Z_0\dasharrow X$ is dominant. Hence $f_{|Z_0}$ is finite, and since $Z_0$ is rational the $3$-fold $X$ is unirational.
\end{proof}
\begin{Theorem}\label{12}
The moduli space $\mathcal{A}_2(1,12)_{sym}^{-}$ of $(1,12)$-polarized abelian surfaces with canonical level structure, a symmetric theta structure and an odd theta characteristic is unirational but not rational.
\end{Theorem}
\begin{proof}
Let $X_4^3$ be the quartic $3$-fold defined by $\{x_1x_3^3+x_3^3x_5-x_2^3x_4-x_2x_4^3+x_1^3x_5+x_1x_5^3 = 0\}\subset\mathbb{P}^{4}_{-}$. An argument analogous to the one used in the proof of Theorem \ref{18} shows that the diagram 
\[
  \begin{tikzpicture}[xscale=2.9,yscale=-1.7]
    \node (A0_0) at (0, 0) {$\mathcal{A}_2(1,12)_{sym}^{-}$};
    \node (A0_1) at (1, 0) {$X_4^3$};
    \node (A1_0) at (0, 1) {$\mathcal{A}_2(1,12)^{lev}$};
    \node (A1_1) at (1, 1) {$X_4^3/\mathbb{Z}_{2}\times\mathbb{Z}_2$};
    \path (A0_0) edge [->]node [auto] {$\scriptstyle{Th_{(1,12)}^{-}}$} (A0_1);
    \path (A0_0) edge [->,swap]node [auto] {$\scriptstyle{f^{-}}$} (A1_0);
    \path (A0_1) edge [->]node [auto] {$\scriptstyle{}$} (A1_1);
    \path (A1_0) edge [->,dashed]node [auto] {$\scriptstyle{\Theta_{12}}$} (A1_1);
  \end{tikzpicture}
  \]
commutes. Since, by Theorem \ref{thm62}, the map $\Theta_{12}$ is birational, the map $Th_{(1,12)}^{-}$ is birational as well. Finally, by Proposition \ref{3uni} we have that $\mathcal{A}_2(1,12)_{sym}^{-}$ is unirational but not rational. 
\end{proof}
\subsubsection{The cases $n=14$ and $n = 16$}\label{1416}
By Theorem \ref{thm62} the map
$$\Theta_{14}:\mathcal{A}_2(1,14)^{lev}\dasharrow\mathbb{P}^{5}_{-}/\mathbb{Z}_{2}\times\mathbb{Z}_{2}$$
mapping an abelian surface $A$ to the orbit of $A\cap\mathbb{P}^{5}_{-}$ is birational onto its image. Let $X_{14}$ be the inverse image of $\Ima(\Theta_{14})$ via the projection $\mathbb{P}^{5}_{-}\rightarrow \mathbb{P}^{5}_{-}/\mathbb{Z}_{2}\times\mathbb{Z}_{2}$. Now, the first $4\times 4$ minor of the matrix $M_7(x,x)$, from equation (\ref{matrame}), restricted to $\mathbb{P}^{5}_{-}$ gives the following pfaffian
$$f =  x_1x_3^3-x_2^3x_4-x_1x_3x_4^2+x_1^3x_5-x_2^2x_3x_5-x_2x_4x_5^2-x_3x_5^3+x_1x_2x_6+x_3^2x_4x_6+x_4^3x_6+x_1x_5x_6^2+x_2x_6^3.$$
On the other hand, the first $4\times 4$ minor of the matrix $M_{7}(\sigma(x),x)$ restricted to $\mathbb{P}^{5}_{-}$ yields the pfaffian
$$g = x_1x_3x_4^2-x_2^2x_3x_5-x_2x_4x_5^2+x_1^2x_2x_6+x_3^2x_4x_6+x_1x_5x_6^2.$$
Clearly, by Theorem \ref{homoidea}, $X_{14}\subseteq X_{4,4}^3 = \{f = g = 0\}\subset\mathbb{P}^{5}_{-}$. Furthermore, a standard computation in \cite{Mc2} shows that $X_{4,4}^3$ is an irreducible $3$-fold of degree $16$ which is singular along a curve of degree $24$. Finally, we get that the map 
$$Th_{14}^{-}:\mathcal{A}_2(1,14)_{sym}^{-}\rightarrow X_{4,4}^3$$
is birational.
The case $n = 16$ is quite similar. By \cite[Lemma 4.1]{GP4} the variety $X_{40}^{3}\subset\mathbb{P}^{6}_{-}$ defined by the $4\times 4$ pfaffians of $M_{8}(x,x)$ is an irreducible $3$-fold of degree $40$. By Theorem \ref{thm62} the map
$$\Theta_{16}:\mathcal{A}_2(1,16)^{lev}\dasharrow \mathbb{P}^{6}_{-}/\mathbb{Z}_{2}\times\mathbb{Z}_{2}$$
is birational onto its image. If $\pi:\mathbb{P}^{6}_{-}\rightarrow\mathbb{P}^{6}_{-}/\mathbb{Z}_{2}\times\mathbb{Z}_{2}$ then $X_{40}^3 = \pi^{-1}(\Ima(\Theta_{16}))$. As usual, we get that the map 
$$Th_{16}^{-}:\mathcal{A}_2(1,16)_{sym}^{-}\rightarrow X_{40}^3$$
is birational. Furthermore, we have the following.
\begin{Proposition}\label{A116}
The moduli space $\mathcal{A}_2(1,16)_{sym}^{-}$ is of general type. 
\end{Proposition}
\begin{proof}
By \cite[Remark 4.2]{GP4} the $3$-fold $X_{40}^3$ is of general type. 
\end{proof}


\begin{thebibliography}{99999999}
\bibitem[BB66]{BB} \bibaut{W. Baily, A. Borel}, \textit{Compactification of arithmetic quotients of bounded symmetric domains}, Ann. of Math. 2, 84, 1966, 442 -- 528.
\bibitem[Ba87]{Ba} \bibaut{W. Barth}, \textit{Abelian surfaces with $(1,2)$-polarization}, Algebraic geometry, Sendai, 1985, 41 -- 84, Adv. Stud. Pure Math., 10, North-Holland, Amsterdam, 1987. 
\bibitem[Be77]{Be} \bibaut{A. Beauville}, \textit{Vari\'et\'es de Prym et jacobiennes interm\'ediaires}, Ann. Sci. \'Ecole Norm. Sup. 4. 10, 1977, no. 3, 309 -- 391.
\bibitem[BL04]{BL} \bibaut{C. Birkenhake, H. Lange}, \textit{Complex Abelian Varieties}, Grundlehren der Mathematischen Wissenschaften, A series of Comprehensive Studies in Mathematics, 32, 2014
\bibitem[Bo08]{Bo1} \bibaut{M. Bolognesi}, \textit{A conic bundle degenerating on the Kummer surface}, Math. Zeit, 261, 2008, 149 -- 168.
\bibitem[Bo07]{Bo} \bibaut{M. Bolognesi}, \textit{On Weddle Surfaces and Their Moduli}, Adv. in  Geom. 7, 2007, no. 1, 113 -- 144. 
\bibitem[BM]{BM} \bibaut{M. Bolognesi, A. Massarenti}, \textit{Moduli of abelian surfaces, symmetric theta structures and theta characteristics II}, in preparation.
\bibitem[Ca92]{Ca} \bibaut{F. Campana}, \textit{Connexit\'e rationnelle des vari\'et\'es de Fano}, Ann. Sci. \'Ecole Norm. Sup. 25, 1992, 539 -- 545.
\bibitem[Cl83]{Cle} \bibaut{H. Clemens}, \textit{Double solids}, Adv. in Math, 47, 1983, no. 2, 107 -- 230.
\bibitem[CM04]{CM} \bibaut{A. Corti, M. Mella}, \textit{Birational geometry of terminal quartic $3$-folds I}, Amer. J. Math. 126, 2004, no. 4, 739 -- 761.
\bibitem[Do04]{Dol} \bibaut{I. Dolgachev}, \textit{Dual homogeneous forms and varieties of power sums}, Milan J. Math, 72, 2004, 163 -- 187.
\bibitem[DL08]{DL} \bibaut{I. Dolgachev, D. Lehavi}, \textit{On isogenous principally polarized abelian surfaces}, Curves and abelian varieties, 51 -- 69,  Contemp. Math, 465, Amer. Math. Soc, Providence, RI, 2008.
\bibitem[DO88]{DO} \bibaut{I. Dolgachev, D. Ortland}, \textit{Point sets in projective spaces and theta functions}, Ast\'erisque, no. 165, 1988, 210 pp.
\bibitem[Don84]{Do} \bibaut{R. Donagi}, \textit{The unirationality of $\A_5$}, Ann. of Math. 2, 119, 1984, no. 2, 269 -- 307.
\bibitem[Er04]{Er} \bibaut{C. Erdenberger}, \textit{The Kodaira dimension of certain moduli spaces of abelian surfaces}, Math. Nachr, 274 -- 275, 2004, 32-39.
\bibitem[vG]{vG} \bibaut{B. van Geemen}, \textit{The moduli space of curves of genus 3 with level 2 structure is rational}, unpublished preprint.
\bibitem[vdG87]{vdG} \bibaut{G. van der Geer}, \textit{Note on abelian schemes of level three}, Math. Ann. 278, 1987, 401 -- 408.
\bibitem[GHS03]{GHS} \bibaut{T. Graber, J. Harris, J. Starr}, \textit{Families of rationally connected varieties}, J. Amer. Math. Soc. 16, 2003, no. 1, 57 -- 67.
\bibitem[GH78]{GH} \bibaut{P. Griffiths, J. Harris}, \textit{Principles of algebraic geometry}, Pure and Applied Mathematics, Wiley-Interscience, New York, 1978, xii+813 pp.
\bibitem[Gr93]{Gri1} \bibaut{V. Gritsenko}, \textit{Irrationality of the moduli spaces of polarized abelian surfaces, with an appendix by the author and K. Hulek}, in Abelian varieties (Egloffstein, 1993), 63 -- 84, de Gruyter, Berlin, 1995.
\bibitem[Gr94]{Gri2} \bibaut{V. Gritsenko}, \textit{Irrationality of the moduli spaces of polarized abelian surfaces}, Internat. Math. Res. Notices, 1994, no. 6, 235, 235 -- 243.
\bibitem[GS96]{GS} \bibaut{V. Gritsenko, G. Sankaran}, \textit{Moduli of abelian surfaces with a $(1,p^2)$ polarisation}, Izv. Ross. Akad. Nauk Ser. Mat. 60, 1996, no. 5, 19 -- 26; translation in Izv. Math. 60 ,1996, no. 5, 893--900.
\bibitem[GH04]{GH1} \bibaut{B. H. Gross, J. Harris}, \textit{On some geometric constructions related to theta characteristics}, Contributions to automorphic forms, geometry, and number theory, 279 -- 311, Johns Hopkins Univ. Press, Baltimore, 2004.
\bibitem[GP01]{GP1} \bibaut{M. Gross, S. Popescu}, \textit{Calabi-Yau Threefolds and Moduli of
Abelian Surfaces I}, Compositio Mathematica, 127, 2001, 169 -- 228.
\bibitem[GP11]{GP4} \bibaut{M. Gross, S. Popescu}, \textit{Calabi-Yau Threefolds and Moduli of
Abelian Surfaces II}, Trans. Amer. Math. Soc. 363, 2011, 3573 -- 3599.
\bibitem[GP98]{GP2} \bibaut{M. Gross, S. Popescu}, \textit{Equations of $(1,d)$-polarized abelian surfaces}, Math. Ann, 310, 1998, 333 -- 377.
\bibitem[GP01/2]{GP3} \bibaut{M. Gross, S. Popescu}, \textit{The Moduli Space of $(1,11)$-polarized abelian surfaces is unirational}, Compositio Mathematica, 126, 2001, 1 -- 24.
\bibitem[HKW93]{HKW} \bibaut{K. Hulek, C. Kahn, S. Weintraub} \textit{Moduli spaces of abelian surfaces: compactification, degenerations, and theta functions}, de Gruyter Expositions in Mathematics, 12, Walter de Gruyter \& Co, Berlin, 1993, xii+347 pp.
\bibitem[HS94]{HS} \bibaut{K. Hulek, G. Sankaran}, \textit{The Kodaira dimension of certain moduli spaces of abelian surfaces}, Compositio Math, 90, 1994, 1 -- 35.
\bibitem[Ig64]{Igu} \bibaut{J. I. Igusa}, \textit{On the graded ring of theta-constants I}, Amer. J. Math. 86, 1964, 219 -- 246.
\bibitem[IR01]{IR} \bibaut{A. Iliev, K. Ranestad}, \textit{$K3$ surfaces of genus $8$ and varieties of sums of powers of cubic fourfolds}, Trans. Amer. Math. Soc, 2001, no. 4, 1455 -- 1468.
\bibitem[IM64]{IM} \bibaut{V. A. Iskovskikh, Y. Manin}, \textit{Three-dimensional quartics and counterexamples to the L\"uroth problem}, Amer. J. Math. 86, 1964, 219 -- 246.
\bibitem[Ka96]{Kat} \bibaut{P. Katsylo}, \textit{Rationality of the moduli variety of curves of genus 3}, Comment. Math. Helv. 71, 1996, no. 4, 507 -- 524.
\bibitem[KL90]{KL} \bibaut{P. Kleidman, M. Liebeck}, \textit{The subgroup structure of the finite classical groups}, London Mathematical Society Lecture Note Series, 129, Cambridge University Press, Cambridge, 1990. x+303 pp.
\bibitem[Ko96]{Ko} \bibaut{J. Koll\'ar}, \textit{Rational Curves on Algebraic Varieties}, Ergebnisse der Mathematik und ihrer Grenzgebiete, Springer-Verlag, Berlin, 1996. 
\bibitem[KMM92]{KMM} \bibaut{J. Koll\'ar, Y. Miyaoka, S. Mori}, \textit{Rational connectedness and boundedness of Fano manifolds}, J. Diff. Geom. 36, 1992, 765 -- 769.
\bibitem[LO11]{LO} \bibaut{J. M. Landsberg, G. Ottaviani}, \textit{Equations for secant varieties of Veronese and other varieties}, Annali di Matematica Pura ed Applicata, vol. 192, Issue 4, 2011 569-606.
\bibitem[Mc2]{Mc2} \bibaut{MacAulay2}, \textit{Macaulay2 a software system devoted to supporting research in algebraic geometry and commutative algebra}, \href{http://www.math.uiuc.edu/Macaulay2/}{http://www.math.uiuc.edu/Macaulay2/}.
\bibitem[MS01]{MS} \bibaut{N. Manolache, F. O. Schreyer}, \textit{Moduli of $(1,7)$-polarized abelian surfaces via syzygies}, Math. Nachr, 226, 2001, 177 -- 203.
\bibitem[Ma16]{Ma} \bibaut{A. Massarenti}, \textit{Generalized varieties of sums of powers}, Bulletin of the Brazilian Mathematical Society, 2016, \href{http://link.springer.com/article/10.1007/s00574-016-0113-6?wt_mc=internal.event.1.SEM.ArticleAuthorOnlineFirst}{DOI: 10.1007/s00574-016-0113-6}.
\bibitem[MM13]{MMe} \bibaut{A. Massarenti, M. Mella}, \textit{Birational aspects of the geometry of Varieties of Sums of Powers}, Advances in Mathematics, 2013, no. 243, 187 -- 202.
\bibitem[Mar04]{Mar} \bibaut{A. Marini}, \textit{On a family of $(1,7)$-polarised abelian surfaces}, Math. Scand. 95, 2004, no. 2, 181 -- 225.
\bibitem[Me04]{Me1} \bibaut{M. Mella}, \textit{Birational geometry of quartic $3$-folds II. The importance of being $\mathbb{Q}$-factorial}, Math. Ann, 330, 2004, no. 1, 107 -- 126.
\bibitem[Me14]{Me} \bibaut{M. Mella}, \textit{On the unirationality of $3$-fold conic bundles}, \arXiv{1403.7055v1}.
\bibitem[MR05]{MR} \bibaut{F. Melliez, K. Ranestad}, \textit{Degenerations of $(1,7)$-polarized abelian surfaces}, Math. Scand. 97, 2005, no. 2, 161 -- 177. 
\bibitem[MM82]{MM} \bibaut{S. Mori, S. Mukai}, \textit{The uniruledness of the moduli space of curves of genus 11}, in Algebraic geometry (Tokyo/Kyoto, 1982), 334 -- 353, Lecture Notes in Math. 1016, Springer, 1983.
\bibitem[Mu92]{Mu1} \bibaut{S. Mukai}, \textit{Fano $3$-folds}, Lond. Math. Soc. Lect. Note Ser. 179, 1992, 255 -- 263.
\bibitem[Mu92/2]{Mu2} \bibaut{S. Mukai}, \textit{Polarized $K3$ surfaces of genus $18$ and $20$}, Complex Projective Geometry, Lond. Math. Soc. Lect. Note Ser, Cambridge University Press, 1992, 264 -- 276.
\bibitem[Mum66]{Mum1} \bibaut{D. Mumford}, \textit{On the equations defining abelian varieties I}, Invent. Math. 1, 1966, 287 -- 354. 
\bibitem[NvG95]{NVG}  \bibaut{N. Nygaard, B. van Geemen}, \textit{On the geometry and arithmetic of some Siegel modular threefolds}, J. Number Theory, 53, 1995, no. 1, 45 -- 87. 
\bibitem[O'G89]{O'G} \bibaut{K. O'Grady}, \textit{On the Kodaira dimension of moduli spaces of abelian surfaces}, Compositio Math. 72, 1989, no. 2, 121 -- 163.
\bibitem[Pa03]{Pa} \bibaut{S. Pavanelli}, \textit{Mirror symmetry for a two parameter family of Calabi-Yau three-folds with Euler characteristic $0$}, Ph.D. thesis, University of Warwick, Mathematics Institute, 2003.
\bibitem[RS00]{RS} \bibaut{K. Ranestad, F. O. Schreyer}, \textit{Varieties of sums of powers}, J. Reine Angew. Math. 525, 2000, 147 -- 181.
\bibitem[Sa76]{Sa76} \bibaut{N. Saavedra Rivano}, \textit{Finite geometries in the theory of theta characteristics}, Enseignement Math, 22, 1976, no. 3-4, 191-218.
\bibitem[SM83]{SM} \bibaut{R. Salvati Manni}, \textit{On the not identically zero Nullwerte of Jacobians of theta functions with odd characteristics}, Adv. in Math, 47, 1983, no. 1, 88 -- 104.
\bibitem[SM85]{SM1} \bibaut{R. Salvati Manni}, \textit{On the not integrally closed subrings of the ring of the thetanullwerte}, Duke Math. J. 52, 1985, no. 1, 25 -- 33.
\bibitem[Sa97]{Sa97} \bibaut{G. Sankaran}, \textit{Moduli of polarised abelian surfaces},
Math. Nachr. 188, 1997, 321 -- 340.  
\bibitem[Se60]{Se} \bibaut{B. Segre}, \textit{Variazione continua ed omotopia in geometria algebrica}, Ann. Mat. Pura Appl. 4, 50, 1960, 149 -- 186.
\bibitem[Ver84]{Ver} \bibaut{A. Verra}, \textit{A short proof of the unirationality of $\A_5$}, Nederl. Akad. Wetensch. Indag. Math. 46, 1984, no. 3, 339 -- 355.
\end{thebibliography}
\end{document}